\documentclass[10pt]{article}
\usepackage{amsmath}
\numberwithin{equation}{section}
\usepackage{amsfonts}
\usepackage{amssymb}
\usepackage{graphicx}
\usepackage{mathrsfs}
\usepackage{xcolor}
\usepackage{verbatim}
\usepackage{mathrsfs}
\usepackage[body={15.5cm,21cm}, top=3cm]{geometry}
\usepackage{paralist}
\usepackage{ntheorem}
\usepackage{appendix}

\DeclareMathOperator*{\esssup}{ess\,sup}
\allowdisplaybreaks[4]
\usepackage{hyperref}
\hypersetup{colorlinks=true,
linkcolor=blue,
anchorcolor=blue,
citecolor=blue}
\providecommand{\U}[1]{\protect \rule{.1in}{.1in}}
\newtheorem{theorem}{Theorem}[section]

\newtheorem{corollary}[theorem]{Corollary}

\newtheorem{definition}[theorem]{Definition}

\newtheorem{lemma}[theorem]{Lemma}

\newtheorem{remark}[theorem]{Remark}

\theoremstyle{empty}

\newenvironment{proof}[1][Proof]{\noindent \textbf{#1.} }{\  \rule{0.5em}{0.5em}}
\allowdisplaybreaks[2]
\begin{document}
\title{Stochastic linear-quadratic optimal control problems with multi-dimensional state and non-Markovian regime switching}
\author{Yuyang Chen\thanks{School of Mathematical Sciences, Shanghai Jiao Tong University, China (cyy0032@sjtu.edu.cn)}
\and
Peng Luo \thanks{School of Mathematical Sciences, Shanghai Jiao Tong University, China (peng.luo@sjtu.edu.cn). Financial support
from the National Natural Science Foundation of China (Grant No. 12101400) is gratefully acknowledged.}}

\maketitle
\begin{abstract}
This paper investigates the stochastic linear-quadratic (LQ, for short) optimal control problems with non-Markovian regime switching in a finite time horizon where the state equation is multi-dimensional. Similar to the classical stochastic LQ problems, we establish the relationship between the stochastic LQ optimal control problems with non-Markovian regime switching and the related extended stochastic Riccati equations. To solve the extended stochastic Riccati equations, we construct a monotone Piccard iterative sequence and present the link bewteen this sequence and solutions of a family of forward-backward stochastic differential equations. Relying on $L^p$ estimates for FBSDEs, we show that the extended stochastic Riccati equation has a solution. This partially addresses one question left in Hu et al. (Ann. Appl. Probab. 32(1): 426-460, 2022). Finally, the stochastic LQ optimal control problems with non-Markovian regime switching is solved.
\end{abstract}

\section{Introduction}
Linear-quadratic (LQ, for short) optimal control problem is one of the most important control problems, because it not only stands out on its own as an interesting mathematically theoretic problem, but also appears in many fields, such as engineering, management science and mathematical finance, see \cite{Chen and Zhu 2018,Li and Ng 2000,Lim and Zhou 2002,Tchamna et al. 2016}.

In 1968, Wonham \cite{Wonham 1968} considered a simple deterministic form of the Riccati equation and connected it with the LQ control problems. Since then, the study of the corresponding Riccati equation has become an important means to solve the stochastic LQ problem. Bismut \cite{Bismut 1976} proposed a more general form in which the coefficients are random but the Riccati equation has an additional martingale part. However, he was only able to prove the existence and uniqueness of the solution in a more simple case. Bensoussan \cite{Bensoussan 1982} derived the various forms of the stochastic maximum principle and showed that the value function of a certain stochastic control problem is the solution of the Hamilton-Jacobi equation. A stochastic Riccati equation with some coefficients degenerating to zero was discussed by Peng \cite{Peng 1992} and he gave the existence and uniqueness result with Picard's method. Chen et al. \cite{Chen et al. 1998} solved a special case in which all of the coefficients are deterministic functions. Kohlmann and Tang \cite{Kohlmann and Tang 2002} referred to the one-dimensional stochastic Riccati equation and solved it with Lepeltier and San Martin's method. They \cite{Kohlmann and Tang 2003a} obtained the solvability results of matrix-valued stochastic Riccati equations in two cases, the so-called singular one and regular one, on the base of Peng's work \cite{Peng 1992}. Kohlmann and Tang \cite{Kohlmann and Tang 2003b} developed a perturbation method to solve the stochastic Riccati equation with bounded coefficients under some extra conditions. Hu and Zhou \cite{Hu and Zhou 2005} discussed a stochastic linear-quadratic optimal control problem where the control is constrained in a cone, which is associated with two extended stochastic Riccati equations. Chen and Yong \cite{Chen and Yong 2001} obtained a sufficient condition and a necessary condition for stochastic linear-quadratic optimal control problems although they were unable to solve the Riccati equations. Rami et al. \cite{Rami et al. 2002} introduced a generalized Riccati equation with a pseudo matrix inverse and an additional algebraic constraint.

The regime switching model well simulates the transition under different states. The combination of stochastic LQ problems and regime switching can be dated back to \cite{Li et al. 2003,Liu et al. 2005,Zhou and Yin 2003}. Recently, the stochastic LQ optimal control problem with regime switching has attracted a lot of interest due to its application in financial management and economics. Zhang et al. \cite{Zhang et al. 2021} discussed the multi-dimensional stochastic LQ problem with Markovian regime switching and obtained the solvability of the matrix-valued Riccati equation. However, their coefficient matrices were deterministic. In \cite{Hu et al. 2022a} and \cite{Hu et al. 2022b}, Hu et al. solved the one-dimensional stochastic LQ problem with non-Markovian regime switching and the associated scalar-valued stochastic Riccati equation, and they applied it to the mean variance asset liability management. We refer readers to \cite{Alia and Alia 2023,Hao et al. 2022,Hu et al. 2022c} for more studies on the stochastic LQ problems with regime switching.

However, multi-dimensional stochastic LQ optimal problem with non-Markovian regime switching is still challenging (see \cite{Hu et al. 2022a}). The objective of this paper is to study this problem. During the preparation of this work, we notice that Wen et al. \cite{Wen et al. 2022} also studied this problem while the title of their paper is misleading. Compare with Wen et al. \cite{Wen et al. 2022}, we provide a different approach based on FBSDE theory. Firstly, we establish the relationship between the stochastic LQ optimal control problems with non-Markovian regime switching and the related extended stochastic Riccati equations. However, the extend stochastic Riccati equation turns to be a family of highly nonlinear matrix-valued BSDEs. To the best of our knowledge, no existing results could be directly used to solve it. To solve this extended stochastic Riccati equation, we first construct a monotone sequence $\left\{(\widetilde{P}_{k},\widetilde{\Lambda}_{k})\right\}_{k\geq0}$ and provide a priori estimates. Next we obtain the convergence of $\left\{\widetilde{P}_{k}\right\}_{k\geq0}$ by monotone convergence theorem and dominated convergence theorem. In order to prove the convergence of $\left\{\widetilde{\Lambda}_{k}\right\}_{k\geq0}$, we need to show that $\left\{\widetilde{P}_{k}\right\}_{k\geq0}$ converges in some finer space. The key ingredient is to link $\left\{(\widetilde{P}_{k},\widetilde{\Lambda}_{k})\right\}_{k\geq1}$ with the solutions of fully coupled FBSDEs $\left\{(\mathbf{X}_k,\mathbf{Y}_k,\mathbf{Z}_k)\right\}_{k\geq 1}$. We establish $L^p$ estimates and obtain convergence of $\left\{(\mathbf{X}_k,\mathbf{Y}_k,\mathbf{Z}_k)\right\}_{k\geq 1}$. Finally we get the convergence of $\left\{(\widetilde{P}_{k},\widetilde{\Lambda}_{k})\right\}_{k\geq0}$ and obtain a solution for the extended stochastic Riccati equation. This partially addresses one question left in Hu et al. (Ann. Appl. Probab. 32(1): 426-460, 2022). Relying on this solvability, we obtain the optimal feedback control of stochastic LQ problem with non-Markovian regime switching.

The rest of the paper is organized as follows. We formulate a multi-dimensional stochastic LQ problem with non-Markovian regime switching and state our main results in Section 2. In Section 3, we provide the existence of solutions for the matrix-valued stochastic Riccati equations with regime switching. Section 4 is denvoted to get the optimal feedback control of the stochastic LQ problem with non-Markovian regime switching.

\section{Formulation of the problem and main results}
Let $(\Omega, \mathcal{F}, \mathbb{P})$ be a fixed complete probability space on which are defined a standard one-dimensional Brownian motion $W=\{W(t);0\leq t<\infty\}$ and a continuous-time stationary Markov chain $\alpha_{t}$ valued in a finite state space $\mathcal{M}=\{1,2, \ldots, \ell\}$ with $\ell>1$. We assume $W(t)$ and $\alpha_{t}$ are independent processes. The Markov chain has a generator $Q=\left(q_{i j}\right)_{\ell \times \ell}$ with $q_{i j} \geq 0$ for $i \neq j$ and $\sum_{j=1}^{\ell} q_{i j}=0$ for every $i \in \mathcal{M}$. Define the filtrations $\mathcal{F}_{t}=\sigma\left\{W(s), \alpha_{s}: 0 \leq s \leq t\right\} \vee \mathcal{N}$ and $\mathcal{F}_{t}^{W}=\sigma\{W(s): 0 \leq s \leq t\} \vee \mathcal{N}$, where $\mathcal{N}$ is the totality of all the $\mathbb{P}$-null sets of $\mathcal{F}$. For a random variable $\eta$, $\|\eta\|_{\infty}$ denotes the $L^{\infty}$-norm of $\eta$, i.e., $\|\eta\|_{\infty}:=\mathop{\esssup}\limits_{\omega}|\eta(\omega)|$. Equalities and inequalities between random variables and processes are understood in the $P$-a.s. and $P\otimes dt$-a.e. sense, respectively.

We use the following notation throughout the paper:
\begin{align*}
&\mathbb{R}^{n}:\text{ the }n\text{-dimensional Euclidean space with the Euclidean norm }|\cdot|;\\
&\mathbb{R}^{m\times n}:\text{ the Euclidean space of all }(m\times n)\text{ real matrices};\\
&\mathbb{S}^{n}:\text{ the space of all symmetric }(n\times n)\text{ real matrices};\\
&I_{n}:\text{ the identity matrix of size }n;\\
&M^{\top}:\text{ the transpose of a matrix }M;\\
&tr(M):\text{ the trace of a matrix }M;\\
&\langle\cdot,\cdot\rangle:\text{ the Frobenius inner product on }\mathbb{R}^{n\times m},\text{ which is defined by }\langle A,B\rangle=tr(A^{\top}B);\\
&|M|:\text{ the Frobenius norm of a matrix }M,\text{ defined by }\left(tr(MM^{\top})\right)^{\frac{1}{2}};
\end{align*}
Furthermore, we introduce the following spaces of random processes: for Euclidean space $\mathbb{H}=\mathbb{R}^{n},\mathbb{R}^{n\times m},\mathbb{S}^{n}$ and $p\geq2$, 
\begin{align*}
&L_{\mathcal{F}}^{\infty}(\Omega;\mathbb{H})=\left\{\xi:\Omega\rightarrow\mathbb{H}\mid\xi\text{ is }\mathcal{F}_{T}\text{-measurable, and essentially bounded }\right\};\\
&\begin{aligned}\mathcal{H}_{\mathcal{F}}^{p}(0,T;\mathbb{H})=
  &\Bigg\{\phi:[0,T]\times \Omega\rightarrow\mathbb{H}\mid\phi(\cdot)\text{ is an }\left\{\mathcal{F}_{t}\right\}_{t\geq0}\text{-adapted process}\\
  &\text{ with }\mathbb{E}\left(\int_{0}^{T}|\phi(t)|dt\right)^{p}<\infty\Bigg\};
  \end{aligned}\\
&\begin{aligned}L_{\mathcal{F}}^{p}(0,T;\mathbb{H})=
&\Bigg\{\phi:[0,T]\times \Omega\rightarrow\mathbb{H}\mid\phi(\cdot)\text{ is an }\left\{\mathcal{F}_{t}\right\}_{t\geq0}\text{-adapted process}\\
&\text{ with }\mathbb{E}\left(\int_{0}^{T}|\phi(t)|^{2}dt\right)^{\frac{p}{2}}<\infty\Bigg\};
\end{aligned}\\
&\begin{aligned}L_{\mathcal{F}}^{p}(\Omega;C(0,T;\mathbb{H}))=&\Bigg\{\phi:[0,T]\times\Omega\rightarrow\mathbb{H}\mid\phi(\cdot)\text{ is an }\left\{\mathcal{F}_{t}\right\}_{t \geq 0}\text{-adapted process}\\
&\text{ and continuous with }\mathbb{E}\left(\sup_{t\in[0,T]}|\phi(t)|^{p}\right)<\infty\Bigg\};\end{aligned}\\
&\begin{aligned}L_{\mathcal{F}}^{\infty}(0,T;\mathbb{H})=&\Bigg\{\phi:[0,T]\times\Omega\rightarrow\mathbb{H}\mid\phi(\cdot)\text{ is an }\left\{\mathcal{F}_{t}\right\}_{t\geq 0}\text{-adapted essentially bounded process}\Bigg\}.\end{aligned}
\end{align*}
$L_{\mathcal{F}^W}^{\infty}(\Omega;\mathbb{H})$, $\mathcal{H}_{\mathcal{F}^W}^{p}(0,T;\mathbb{H})$, $L_{\mathcal{F}^W}^{p}(0,T;\mathbb{H})$, $L_{\mathcal{F}^W}^{p}(\Omega;C(0,T;\mathbb{H}))$ and $L_{\mathcal{F}^W}^{\infty}(0,T;\mathbb{H})$ are defined in a same manner by replacing $\mathcal{F}$ by $\mathcal{F}^{W}$.

We now introduce the multi-dimensional stochastic LQ optimal problem with non-Markovian regime switching. Consider the following $n$-dimensional controlled linear stochastic differential equation on the finite time interval $[0,T]$:
\begin{equation}\label{state}
\left\{\begin{array}{l}
dX(t)=\left[A\left(t,\alpha_{t}\right)X(t)+B\left(t,\alpha_{t}\right)u(t)\right]dt+\left[C\left(t,\alpha_{t}\right)X(t)+D\left(t,\alpha_{t}\right)u(t)\right]dW(t),~t\in[0,T],\\
X(0)=x, \alpha_{0}=i_{0},
\end{array}\right.
\end{equation}
where $A(t,\omega,i),B(t,\omega,i),C(t,\omega,i),D(t,\omega,i)$ are all $\left\{\mathcal{F}^{W}_{t}\right\}_{t\geq 0}$-adapted processes of suitable sizes for $i \in \mathcal{M}$ and $x\in\mathbb{R}^{n}$ is an initial state, $i_{0}\in\mathcal{M}$ is an initial regime. The solution $X=\{X(t);0\leq t\leq T\}$ of \eqref{state}, valued in $\mathbb{R}^{n}$, is called a state process; the process $u=\{u(t);0\leq t\leq T\}$ of \eqref{state}, valued in $\mathbb{R}^{m}$, is called a control which influences the state $X$, and is taken from the space $\mathcal{U}[0,T]:=L_{\mathcal{F}}^{2}(0,T;\mathbb{R}^{m})$.

In order to measure the performance of control $u(\cdot)$, we introduce the following quadratic cost functional:
\begin{equation}\label{cost}
\begin{aligned}
J\left(x, i_{0}, u(\cdot)\right):=\mathbb{E}\left[\langle G(\alpha_{T})X(T), X(T)\rangle+\int_{0}^{T}\left\langle\left(\begin{array}{cc}
Q(t,\alpha_{t}) & S(t,\alpha_{t})^{\top} \\
S(t,\alpha_{t}) & R(t,\alpha_{t})
\end{array}\right)\left(\begin{array}{c}
X(t) \\
u(t)
\end{array}\right),\left(\begin{array}{c}
X(t) \\
u(t)
\end{array}\right)\right\rangle dt\right].
\end{aligned}
\end{equation}
For state equation \eqref{state} and cost functional \eqref{cost}, we introduce the following assumption:\\
$(\mathscr{A}1)$ For all $i \in \mathcal{M}$,
$$
\left\{\begin{array}{ll}
A(t,\omega,i),~C(t,\omega,i)\in L_{\mathcal{F}^{W}}^{\infty}(0,T;\mathbb{R}^{n\times n
}),\\
B(t,\omega,i),~D(t,\omega,i)\in L_{\mathcal{F}^{W}}^{\infty}(0,T;\mathbb{R}^{n\times m}),\\
Q(t,\omega,i)\in L_{\mathcal{F}^{W}}^{\infty}(0,T;\mathbb{S}^{n}),\\
S(t,\omega,i)\in L_{\mathcal{F}^{W}}^{\infty}(0,T;\mathbb{R}^{m\times n}),\\
R(t,\omega,i)\in L_{\mathcal{F}^{W}}^{\infty}(0,T;\mathbb{S}^{m}),\\
G(\omega,i)\in L_{\mathcal{F}^{W}}^{\infty}(\Omega;\mathbb{S}^{n}).
\end{array}\right.
$$

Under condition $(\mathscr{A}1)$, for any initial state $x$ and any control $u(\cdot)\in\mathcal{U}[0,T]$, standard SDE theory shows that equation \eqref{state} has a unique solution $X(\cdot)\in L_{\mathcal{F}}^{2}(\Omega;C(0,T;\mathbb{R}^{n}))$. We call such $(X(\cdot), u(\cdot))$ an admissible pair.

Then the following problem, called stochastic linear-quadratic optimal control problem with regime switching, can be formulated.\\
\\
$\mathbf{Problem(SLQ)}$. For any initial pair $\left(x, i_{0}\right)\in\mathbb{R}^{n}\times\mathcal{M}$, find a control $u^{*}\in\mathcal{U}[0,T]$ such that
\begin{equation}\label{SLQ}
J\left(x,i_{0},u^{*}\right)=\inf_{u\in\mathcal{U}[0,T]}J\left(x,i_{0},u\right)\equiv V\left(x,i_{0}\right).
\end{equation}
Any element $u^{*}\in\mathcal{U}[0,T]$ satisfying \eqref{SLQ} is called an optimal control of Problem(SLQ) corresponding to the initial pair $\left(x, i_{0}\right)\in\mathbb{R}^{n}\times\mathcal{M}$, the corresponding state process $X^{*}(\cdot)\equiv X(\cdot;u^{*})$ is called an optimal state process. We also call $V\left(x,i_{0}\right)$ the value function of Problem(SLQ). Our objective is to solve Problem(SLQ).

\subsection{The extended stochastic Riccati equation}
Inspired by the relationship between the stochastic LQ problem and a stochastic Riccati equation (see \cite{Peng 1992,Sun and Yong 2020,Sun et al. 2021}), we first study the following extended stochastic Riccati equation (ESRE, for short) with regime switching:
\begin{equation}\label{Riccati}
\left\{
\begin{array}{l}
\begin{aligned}
dP(t,i)=&-\left[P(t,i)A(t,i)+A(t,i)^{\top}P(t,i)+C(t,i)^{\top}P(t,i)C(t,i)+\Lambda(t,i) C(t,i)\right.\\
&\quad+C(t,i)^{\top}\Lambda(t,i)+Q(t,i)+\textstyle\sum_{j=1}^{l}q_{ij}P(t,j)-\left(P(t,i)B(t,i)+C(t,i)^{\top}P(t,i)D(t,i)\right.\\
&\quad\left.+\Lambda(t,i)D(t,i)+S(t,i)^{\top}\right)\left(R(t,i)+D(t,i)^{\top}P(t,i)D(t,i)\right)^{-1}\left(B(t,i)^{\top}P(t,i)\right.\\
&\quad\left.\left.+D(t,i)^{\top}P(t,i)C(t,i)+D(t,i)^{\top}\Lambda(t,i)+S(t,i)\right)\right]dt+\Lambda(t,i) dW(t),
\end{aligned}\\
R(t,i)+D(t,i)^{\top}P(t,i)D(t,i)>0,\\
P(T,i)=G(i),~t\in[0,T],~i\in\mathcal{M}.
\end{array}
\right.
\end{equation}

\begin{definition}
A vector process $(P(i),\Lambda(i))_{i=1}^{\ell}$ is called a solution of ESRE \eqref{Riccati}, if it satisfies \eqref{Riccati}, and $(P(i),\Lambda(i))\in L_{\mathcal{F}^{W}}^{2}(\Omega;C(0,T;\mathbb{S}^{n}))\times L_{\mathcal{F}^{W}}^{2}(0,T;\mathbb{S}^{n})$ for all $i\in\mathcal{M}$.
\end{definition}

For the classical LQ problem, one can see that the positive definite coefficients in \eqref{cost} is usually a sufficient condition for the solvability of Riccati equation and then LQ problem, for example \cite{Tang 2003}. Therefore, we also introduce the following assumption.

$(\mathscr{A}2)$ For all $i \in \mathcal{M},t\in[0,T]$ and some $\delta>0$,
$$
R(t,i)\geq\delta I_{m},\quad Q(t,i)-S(t,i)^{\top}R(t,i)^{-1}S(t,i)\geq0,\quad G(i)\geq0.
$$

\begin{remark}
In fact, with the bounded coefficients, solution to ESRE \eqref{Riccati} belongs to a nicer space, i.e., for all $i\in\mathcal{M}$, $P(i)\in L_{\mathcal{F}^{W}}^{\infty}(0,T;\mathbb{S}^{n})$ and $\Lambda(i)$ is in the class of martingales of bounded mean oscillation, briefly called BMO martingales.
\end{remark}

Now from \cite{Kazamaki 1994}, we denote that the process $\int_{0}^{t}\phi(s)dW(s)$ is a BMO martingale if and only if there exists a constant $c>0$ such that
$$
\mathbb{E}\left[\int_{\tau}^{T}|\phi(s)|^{2}ds\mid\mathcal{F}_{\tau}^{W}\right]\leq c
$$
for all $\left\{\mathcal{F}_{t}^{W}\right\}_{t\geq 0}$-stopping times $\tau\leq T$. The Dol\'eans-Dade stochastic exponential
$$
\mathcal{E}\left(\int_{0}^{t}\phi(s)dW(s)\right)
$$
of a BMO martingale $\int_{0}^{t}\phi(s)dW(s)$ is a uniformly integrable martingale and can be denoted by $\mathcal{E}_{t}(\phi\cdot W)$. For $0\leq s<t<\infty$, $\mathcal{E}_{s,t}(\phi\cdot W)$ denotes the stochastic exponential on the time index $[s,t]$. We also denote
$$
L_{\mathcal{F}^{W}}^{2,\mathrm{BMO}}(0,T;\mathbb{H})=\left\{\phi\in L_{\mathcal{F}^{W}}^{2}\left(0,T;\mathbb{H}\right)\mid\int_{0}^{t}\phi(s)dW(s)\text { is a BMO martingale on }[0, T]\right\}.
$$

If there is no ambiguity, we write either $\mathbb{E}[\cdot|\mathcal{F}_{t}]$ or $\mathbb{E}[\cdot|\mathcal{F}^{W}_{t}]$ as $\mathbb{E}_{t}[\cdot]$.

We are now ready to state our main results.

\subsection{Main results}
ESRE \eqref{Riccati} is a system of matrix-valued BSDEs with highly nonlinear generators, whose solvability is still challenging (see Hu et al. \cite{Hu et al. 2022a}). Our first main result gives the solvability of ESRE \eqref{Riccati} under certain conditions.
\begin{theorem}\label{th-a}
Let $(\mathscr{A}1)$ and $(\mathscr{A}2)$ hold. There exists a sufficiently small constant $L_{\sigma}>0$, when
\begin{equation}\label{condition}
e^{-q_{ii}t}|D(t,i)R(t,i)^{-1}D(t,i)^{\top}|\leq L_{\sigma},\quad i\in\mathcal{M},~t\in[0,T],
\end{equation}
ESRE \eqref{Riccati} has a solution $\left(P(i),\Lambda(i)\right)_{i=1}^{l}$ such that $\left(P(i),\Lambda(i)\right)\in L_{\mathcal{F}^{W}}^{\infty}(0,T;\mathbb{S}^{n})\times L_{\mathcal{F}^{W}}^{2,\mathrm{BMO}}(0,T;\mathbb{S}^{n})$ and $P(i)\geq0$ for $i\in\mathcal{M}$.
\end{theorem}

Relying on the solvability of ESRE \eqref{Riccati}, we are able to solve Problem (SLQ). Our second main result gives an optimal feedback control and an optimal value for Problem (SLQ).
\begin{theorem}\label{th-b}
Under the conditions of Theorem \ref{th-a}, Problem (SLQ) admits an optimal control, as a feedback function of the time $t$, the market regime $i$ and the state $X$,
\begin{equation}\label{optimal}
\begin{aligned}
u^{*}(t,i,X)=&-\left(R(t,i)+D(t,i)^{\top}P(t,i)D(t,i)\right)^{-1}\\
&\cdot\left(B(t,i)^{\top}P(t,i)+D(t,i)^{\top}P(t,i)C(t,i)+D(t,i)^{\top}\Lambda(t,i)+S(t,i)\right)X.
\end{aligned}
\end{equation}
Moreover, the corresponding optimal value is
\begin{align*}
V(x, i_{0})=\langle P(0,i_{0})x,x\rangle
\end{align*}
where $\left(P(i),\Lambda(i)\right)_{i=1}^{l}$ is the solution of ESRE \eqref{Riccati} obtained in Theorem \ref{th-a}.
\end{theorem}

It is obvious that when $D(i)\equiv0$, \eqref{condition} always holds for any $L_{\sigma}>0$. Therefore we have the following corollary.
\begin{corollary}
Under $(\mathscr{A}1)$ and $(\mathscr{A}2)$, the following ESRE:
\begin{equation}\label{Riccati-D}
\left\{
\begin{array}{l}
\begin{aligned}
dP(t,i)=&-\left[P(t,i)A(t,i)+A(t,i)^{\top}P(t,i)+C(t,i)^{\top}P(t,i)C(t,i)+\Lambda(t,i) C(t,i)\right.\\
&\quad+C(t,i)^{\top}\Lambda(t,i)+Q(t,i)+\textstyle\sum_{j=1}^{l}q_{ij}P(t,j)-\left(P(t,i)B(t,i)+S(t,i)^{\top}\right)\\
&\quad\cdot\left.R(t,i)^{-1}\left(B(t,i)^{\top}P(t,i)+S(t,i)\right)\right]dt+\Lambda(t,i) dW(t),~t\in[0,T],
\end{aligned}\\
P(T,i)=G(i),~i\in\mathcal{M},
\end{array}
\right.
\end{equation}
has a solution $\left(P(i),\Lambda(i)\right)_{i=1}^{l}$ such that $\left(P(i),\Lambda(i)\right)\in L_{\mathcal{F}^{W}}^{\infty}(0,T;\mathbb{S}^{n})\times L_{\mathcal{F}^{W}}^{2,\mathrm{BMO}}(0,T;\mathbb{S}^{n})$ for all $i\in\mathcal{M}$. And the cost functional \eqref{cost} associated with the following state equation:
\begin{equation*}
\left\{\begin{array}{l}
dX(t)=\left[A\left(t,\alpha_{t}\right)X(t)+B\left(t,\alpha_{t}\right)u(t)\right]dt+C\left(t,\alpha_{t}\right)X(t)dW(t),~t\in[0,T],\\
X(0)=x, \alpha_{0}=i_{0},
\end{array}\right.
\end{equation*}
is minimized by an optimal feedback control
\begin{equation*}
u^{*}(t,i,X)=-R(t,i)^{-1}\left(B(t,i)^{\top}P(t,i)+S(t,i)\right)X.
\end{equation*}
\end{corollary}

\section{Solvability of the ESRE with regime switching}
In this section, we focus on the solvability of ESRE \eqref{Riccati}. The first key ingredient of our approach is the construction of a monotone Piccard iterative sequence, where the initial element is carefully chosen by an auxiliary BSDE introduced in next subsection.
\subsection{Auxiliary backward stochastic differential equations}
First of all, we recall the following lemma from standard matrix analysis, which is a direct consequence of \cite[Theorem 7.4.1.1]{Horn and Johnson 2012}. We will use this lemma from time to time.

\begin{lemma}\label{le-a}
Let $\mathbf{A},\mathbf{B}\in\mathbb{S}^{n}$ with $\mathbf{B}$ being positive semi-definite. Then with $\lambda_{max}(\mathbf{A})$ denoting the largest eigenvalue of $\mathbf{A}$, we have
$$
tr(\mathbf{A}\mathbf{B})\leq\lambda_{max}(\mathbf{A})\cdot tr(\mathbf{B}).
$$
\end{lemma}

Now we consider a linear backward stochastic differential equation with regime switching:
\begin{equation}\label{P0}
\left\{
\begin{array}{l}
\begin{aligned}dP_{0}(t,i)=-[&P_{0}(t,i)A(t,i)+A(t,i)^{\top}P_{0}(t,i)+C(t,i)^{\top}P_{0}(t,i)C(t,i)+\Lambda_{0}(t,i)C(t,i)\\
&+C(t,i)^{\top}\Lambda_{0}(t,i)+Q(t,i)+\textstyle\sum_{j=1}^{l}q_{ij}P_{0}(t,j)]dt+\Lambda_{0}(t,i)dW(t),~t\in[0,T],
\end{aligned}\\
P_{0}(T,i)=G(i),~i\in\mathcal{M}.
\end{array}
\right.
\end{equation}
For simplicity, throughout this paper, we denote
\begin{equation}\label{notation}
\begin{aligned}
&\Pi(t,i,p,\lambda)=p(t,i)A(t,i)+A(t,i)^{\top}p(t,i)+C(t,i)^{\top}p(t,i)C(t,i)+\lambda(t,i)C(t,i)+C(t,i)^{\top}\lambda(t,i),\\
&H(t,i,p,\lambda,R,S)=-\left(p(t,i)B(t,i)+C(t,i)^{\top}p(t,i)D(t,i)+\lambda(t,i)D(t,i)+S(t,i)^{\top}\right)\\
&\quad\cdot\left(R(t,i)+D(t,i)^{\top}p(t,i)D(t,i)\right)^{-1}\left(B(t,i)^{\top}p(t,i)+D(t,i)^{\top}p(t,i)C(t,i)+D(t,i)^{\top}\lambda(t,i)+S(t,i)\right),\\
&\qquad\qquad\qquad\qquad\qquad\qquad\qquad\qquad\qquad\qquad\qquad\qquad\qquad\qquad\qquad\qquad\qquad\qquad i\in\mathcal{M},~t\in[0,T].
\end{aligned}
\end{equation}
Under $(\mathscr{A}1)$, there exists a constant $K>0$ such that for $i,j\in\mathcal{M},t\in[0,T]$,
\begin{equation}\label{K}
|\Pi(t,i,p,\lambda)|\leq K|p(t,i)|+K|\lambda(t,i)|,\quad|Q(t,i)|\leq K,\quad|G(i)|\leq K,
\end{equation}
\begin{equation}\label{K'}
|q_{ij}e^{(q_{ii}-q_{jj})t}|\leq K.
\end{equation}
We first give a priori estimate for the solution of BSDE \eqref{P0}.
\begin{lemma}\label{le-estimate}
  Assume that  $(\mathscr{A}1)$ holds and BSDE \eqref{P0} adimits a solution $(P_{0}(i),\Lambda_{0}(i))_{i=1}^{l}$ such that $(P_{0}(i),\Lambda_{0}(i))\in L_{\mathcal{F}^{W}}^{\infty}(0,T;\mathbb{S}^{n})\times L_{\mathcal{F}^{W}}^{2,\mathrm{BMO}}(0,T;\mathbb{S}^{n})$ for $i\in\mathcal{M}$, then it holds that all $i\in\mathcal{M}$
  \begin{align*}
    &\|e^{\frac{1}{2}\rho \cdot}\widetilde{P}_0(\cdot,i)\|^2_{L_{\mathcal{F}^{W}}^{\infty}(0,T;\mathbb{S}^{n})}\leq\frac{3}{2}e^{\rho T}(K^{2}+\frac{1}{\rho}),\\
& \|e^{\frac{1}{2}\rho \cdot}\widetilde{\Lambda}_0(\cdot,i)\|^2_{L_{\mathcal{F}^{W}}^{2,\mathrm{BMO}}(0,T;\mathbb{S}^{n})}\leq3e^{\rho T}(K^{2}+\frac{1}{\rho}),
  \end{align*}
  where $\rho=\left(3(l-1)^{2}T+3\right)K^{2}+2K$,   $\widetilde{P}_{0}(t,i)=e^{q_{ii}t}P_{0}(t,i)$, and $\widetilde{\Lambda}_{0}(t,i)=e^{q_{ii}t}\Lambda_{0}(t,i)$ for all $t\in[0,T]$.
  \end{lemma}
  \begin{proof}
    Let $\widetilde{P}_{0}(t,i)=e^{q_{ii}t}P_{0}(t,i)$, and $\widetilde{\Lambda}_{0}(t,i)=e^{q_{ii}t}\Lambda_{0}(t,i)$ for all $i\in\mathcal{M}$ and $t\in[0,T]$. It is obvious that $(\widetilde{P}_{0}(i),\widetilde{\Lambda}_{0}(i))_{i=1}^{l}$ satisfies
    \begin{equation*}
    \left\{
    \begin{array}{l}
    \begin{aligned}
    d\widetilde{P}_{0}(t,i)=&-[\Pi(t,i,\widetilde{P}_{0},\widetilde{\Lambda}_{0})+\widetilde{Q}(t,i)+\textstyle\sum_{j\neq i}q_{ij}e^{(q_{ii}-q_{jj})t}\widetilde{P}_{0}(t,j)]dt+\widetilde{\Lambda}_{0}(t,i)dW(t),~t\in[0,T],
    \end{aligned}\\
    \widetilde{P}_{0}(T,i)=\widetilde{G}(i),~i\in\mathcal{M}.
    \end{array}
    \right.
    \end{equation*} 
Applying It\^o's formula to $e^{\rho t}|\widetilde{P}_{0}(t,i)|^{2}$ for a constant $\rho$ yields
    \begin{equation}\label{rho}
    \begin{aligned}
    e^{\rho t}|\widetilde{P}_{0}(t,i)|^{2}=&e^{\rho T}|\widetilde{G}(i)|^{2}+\int_{t}^{T}e^{\rho s}\left[-\rho|\widetilde{P}_{0}(s,i)|^{2}+2\langle \widetilde{P}_{0}(s,i),\Pi(s,i,\widetilde{P}_{0},\widetilde{\Lambda}_{0})+\widetilde{Q}(s,i)\right.\\
    &\left.+\textstyle\sum_{j\neq i}q_{ij}e^{(q_{ii}-q_{jj})s}\widetilde{P}_{0}(s,j)\rangle-|\widetilde{\Lambda}_{0}(s,i)|^{2}\right]ds-2\int_{t}^{T}e^{\rho s}\langle \widetilde{P}_{0}(s,i),\widetilde{\Lambda}_{0}(s,i)\rangle dW(s).
    \end{aligned}
    \end{equation}
    Taking conditional expectation, it holds that
    \begin{align*}
    &e^{\rho t}|\widetilde{P}_{0}(t,i)|^{2}+\rho \mathbb{E}_{t}\left[\int_{t}^{T}e^{\rho s}|\widetilde{P}_{0}(s,i)|^{2}ds\right]+\mathbb{E}_{t}\left[\int_{t}^{T}e^{\rho s}|\widetilde{\Lambda}_{0}(s,i)|^{2}ds\right]\\
    \leq &\mathbb{E}_{t}\left[e^{\rho T}|\widetilde{G}(i)|^{2}\right]+2K\mathbb{E}_{t}\left[\int_{t}^{T}e^{\rho s}|\widetilde{P}_{0}(s,i)|\left(|\widetilde{P}_{0}(s,i)|+|\widetilde{\Lambda}_{0}(s,i)|+1+\textstyle\sum_{j\neq i}|\widetilde{P}_{0}(s,j)|\right)ds\right]\\
    \leq &\mathbb{E}_{t}\left[e^{\rho T}|\widetilde{G}(i)|^{2}\right]+2K\mathbb{E}_{t}\left[\int_{t}^{T}e^{\rho s}|\widetilde{P}_{0}(s,i)|^{2}ds\right]+2K^{2}\mathbb{E}_{t}\left[\int_{t}^{T}e^{\rho s}|\widetilde{P}_{0}(s,i)|^{2}ds\right]\\
    &+\frac{1}{2}\mathbb{E}_{t}\left[\int_{t}^{T}e^{\rho s}|\widetilde{\Lambda}_{0}(s,i)|^{2}ds\right]+K^{2}\mathbb{E}_{t}\left[\int_{t}^{T}e^{\rho s}|\widetilde{P}_{0}(s,i)|^{2}ds\right]+\mathbb{E}_{t}\left[\int_{t}^{T}e^{\rho s}ds\right]\\
    &+3K^{2}(l-1)^{2}T\mathbb{E}_{t}\left[\int_{t}^{T}e^{\rho s}|\widetilde{P}_{0}(s,i)|^{2}ds\right]+\frac{1}{3(l-1)T}\textstyle\sum_{j\neq i}\mathbb{E}_{t}\left[\int_{t}^{T}e^{\rho s}|\widetilde{P}_{0}(s,j)|^{2}ds\right]\\
    \leq &e^{\rho T}(K^{2}+\frac{1}{\rho})+\frac{1}{3(l-1)T}\textstyle\sum_{j\neq i}\mathbb{E}_{t}\left[\int_{t}^{T}e^{\rho s}|\widetilde{P}_{0}(s,j)|^{2}ds\right]\\
    &+\left(\left(3(l-1)^{2}T+3\right)K^{2}+2K\right)\mathbb{E}_{t}\left[\int_{t}^{T}e^{\rho s}|\widetilde{P}_{0}(s,i)|^{2}ds\right]+\frac{1}{2}\mathbb{E}_{t}\left[\int_{t}^{T}e^{\rho s}|\widetilde{\Lambda}_{0}(s,i)|^{2}ds\right].
    \end{align*}
    Taking $\rho=\left(3(l-1)^{2}T+3\right)K^{2}+2K$, we get 
    \begin{align*}
      &e^{\rho t}|\widetilde{P}_{0}(t,i)|^{2}+\frac{1}{2}\mathbb{E}_{t}\left[\int_{t}^{T}e^{\rho s}|\widetilde{\Lambda}_{0}(s,i)|^{2}ds\right]\\
      \leq &e^{\rho T}(K^{2}+\frac{1}{\rho})+\frac{1}{3(l-1)T}\textstyle\sum_{j\neq i}\mathbb{E}_{t}\left[\int_{t}^{T}e^{\rho s}|\widetilde{P}_{0}(s,j)|^{2}ds\right],
      \end{align*}
    which implies 
    \begin{equation*}
      \|e^{\frac{1}{2}\rho \cdot}\widetilde{P}_0(\cdot,i)\|^2_{L_{\mathcal{F}^{W}}^{\infty}(0,T;\mathbb{S}^{n})}\leq e^{\rho T}(K^{2}+\frac{1}{\rho})+\frac{1}{3(l-1)}\textstyle\sum_{j\neq i}\|e^{\frac{1}{2}\rho \cdot}\widetilde{P}_0(\cdot,j)\|^2_{L_{\mathcal{F}^{W}}^{\infty}(0,T;\mathbb{S}^{n})}.
    \end{equation*}
    Hence, we obtain that
    \begin{equation*}
      \sum_{i=1}^l\|e^{\frac{1}{2}\rho \cdot}\widetilde{P}_0(\cdot,i)\|^2_{L_{\mathcal{F}^{W}}^{\infty}(0,T;\mathbb{S}^{n})}\leq \frac{3l}{2}e^{\rho T}(K^{2}+\frac{1}{\rho}),
    \end{equation*}
    and for all $i\in\mathcal{M}$,
    \begin{equation*}
      \|e^{\frac{1}{2}\rho \cdot}\widetilde{P}_0(\cdot,i)\|^2_{L_{\mathcal{F}^{W}}^{\infty}(0,T;\mathbb{S}^{n})}\leq \frac{3}{2}e^{\rho T}(K^{2}+\frac{1}{\rho}).
    \end{equation*}
    As a byproduct, we further have for all $i\in\mathcal{M}$,
    \begin{equation*}
      \|e^{\frac{1}{2}\rho \cdot}\widetilde{\Lambda}_0(\cdot,i)\|^2_{L_{\mathcal{F}^{W}}^{2,\mathrm{BMO}}(0,T;\mathbb{S}^{n})}\leq3e^{\rho T}(K^{2}+\frac{1}{\rho}).
    \end{equation*}
    \end{proof}

The following lemma provides the solvability result for BSDE \eqref{P0}.

\begin{lemma}\label{le-b}
Let $(\mathscr{A}1)$ holds and $Q(t,i),G(i)\geq0$ for $i\in\mathcal{M},t\in[0,T]$. Then BSDE \eqref{P0} adimits a unique solution $(P_{0}(i),\Lambda_{0}(i))_{i=1}^{l}$ such that $(P_{0}(i),\Lambda_{0}(i))_{i=1}^{l}\in L_{\mathcal{F}^{W}}^{\infty}(0,T;\mathbb{S}^{n})\times L_{\mathcal{F}^{W}}^{2,\mathrm{BMO}}(0,T;\mathbb{S}^{n})$ and $P_{0}(i)\geq0$ for $i\in\mathcal{M}$.
\end{lemma}
\begin{proof}
Given $\widetilde{p}_{0}(j)\in L_{\mathcal{F}^{W}}^{\infty}(0,T;\mathbb{S}^{n})$ satisfying $\widetilde{p}_{0}(j)\geq0$ for all $j\in\mathcal{M}$, it follows from \cite[Theorem 5.1]{Peng 1992} and \cite{Hu et al. 2022a} that the following BSDE
\begin{equation}\label{map}
\left\{
\begin{array}{l}
\begin{aligned}
d\widetilde{P}_{0}(t,i)=&-[\Pi(t,i,\widetilde{P}_{0},\widetilde{\Lambda}_{0})+\widetilde{Q}(t,i)+\textstyle\sum_{j\neq i}q_{ij}e^{(q_{ii}-q_{jj})t}\widetilde{p}_{0}(t,j)]dt+\widetilde{\Lambda}_{0}(t,i)dW(t),~t\in[0,T],
\end{aligned}\\
\widetilde{P}_{0}(T,i)=\widetilde{G}(i),~i\in\mathcal{M},
\end{array}
\right.
\end{equation} 
with
\begin{equation*}
\widetilde{Q}(t,i)=e^{q_{ii}t}Q(t,i),\quad\widetilde{G}(i)=e^{q_{ii}T}G(i),
\end{equation*}
admits a unique solution $(\widetilde{P}_{0}(i),\widetilde{\Lambda}_{0}(i))_{i=1}^{l}$ such that $(\widetilde{P}_{0}(i),\widetilde{\Lambda}_{0}(i))_{i=1}^{l}\in L_{\mathcal{F}^{W}}^{\infty}(0,T;\mathbb{S}^{n})\times L_{\mathcal{F}^{W}}^{2,\mathrm{BMO}}(0,T;\mathbb{S}^{n})$ and $\widetilde{P}_{0}(i)\geq0,~i\in\mathcal{M}$. For $\underline{p}_{0},\overline{p}_{0}\in L_{\mathcal{F}^{W}}^{\infty}(0,T;\mathbb{S}^{n})$ satisfying $\underline{p}_{0},\overline{p}_{0}\geq0$, let $(\underline{P}_{0},\underline{\Lambda}_{0})$ (resp. $(\overline{P}_{0},\overline{\Lambda}_{0})$) be the unique solution of BSDE \eqref{map} corresponding to $\underline{p}_{0}$ (resp. $\overline{p}_{0}$). We denote
$$
\begin{aligned}
&\Delta p_{0}(t,i)=\underline{p}_{0}(t,i)-\overline{p}_{0}(t,i),\quad\Delta\lambda_{0}(t,i)=\underline{\lambda}_{0}(t,i)-\overline{\lambda}_{0}(t,i),\\
&\Delta P_{0}(t,i)=\underline{P}_{0}(t,i)-\overline{P}_{0}(t,i),\quad\Delta\Lambda_{0}(t,i)=\underline{\Lambda}_{0}(t,i)-\overline{\Lambda}_{0}(t,i),\qquad t\in[0,T],~i\in\mathcal{M}.
\end{aligned}
$$
Applying It\^o's formula to $e^{\rho t}|\Delta P_{0}(t,i)|^{2}$ for a constant $\rho$ and taking conditional expectation, we have
\begin{align*}
&e^{\rho t}|\Delta P_{0}(t,i)|^{2}+\rho \mathbb{E}_{t}\left[\int_{t}^{T}e^{\rho s}|\Delta P_{0}(s,i)|^{2}ds\right]+\mathbb{E}_{t}\left[\int_{t}^{T}e^{\rho s}|\Delta\Lambda_{0}(s,i)|^{2}ds\right]\\
=&2\mathbb{E}_{t}\left[\int_{t}^{T}e^{\rho s}\langle\Delta P_{0}(s,i),\Pi(s,i,\Delta P_{0},\Delta\Lambda_{0})\rangle+\sum_{j\neq i}q_{ij}e^{(q_{ii}-q_{jj})s}\Delta\widetilde{p}_{0}(s,j) ds\right]\\
\leq &2K\mathbb{E}_{t}\left[\int_{t}^{T}e^{\rho s}|\Delta P_{0}(s,i)|^{2}ds\right]+2K^{2}\mathbb{E}_{t}\left[\int_{t}^{T}e^{\rho s}|\Delta P_{0}(s,i)|^{2}ds\right]+\frac{1}{2}\mathbb{E}_{t}\left[\int_{t}^{T}e^{\rho s}|\Delta\Lambda_{0}(s,i)|^{2}ds\right]\\
&+3K^{2}(l-1)^{2}T\mathbb{E}_{t}\left[\int_{t}^{T}e^{\rho s}|\Delta P_{0}(s,i)|^{2}ds\right]+\frac{1}{3(l-1)T}\sum_{j\neq i}\mathbb{E}_{t}\left[\int_{t}^{T}e^{\rho s}|\Delta p_{0}(s,j)|^{2}ds\right]\\
\leq &\left\{[3(l-1)^{2}T+2]K^{2}+2K\right\}\mathbb{E}_{t}\left[\int_{t}^{T}e^{\rho s}|{P}_{0}(s,i)|^{2}ds\right]+\frac{1}{2}\mathbb{E}_{t}\left[\int_{t}^{T}e^{\rho s}|\Delta\Lambda_{0}(s,i)|^{2}ds\right]\\
&+\frac{1}{3(l-1)}\sum_{j\neq i}\left\|\sup\limits_{s\in[t,T]}e^{\rho s}|\Delta p(s,i)|^{2}\right\|_{\infty},
\end{align*}
which implies
\begin{align*}
\left\|\sup\limits_{t\in[0,T]}e^{\rho t}|\Delta P_{0}(t,i)|^{2}\right\|_{\infty}
\leq\frac{1}{3(l-1)}\sum_{j\neq i}\left\|\sup\limits_{t\in[0,T]}e^{\rho t}|\Delta p_{0}(t,j)|^{2}\right\|_{\infty}.
\end{align*}
Thus, it holds that
\begin{align*}
\sum_{i=1}^{l}\left\|\sup\limits_{t\in[0,T]}e^{\rho t}|\Delta P_{0}(t,i)|^{2}\right\|_{\infty}\leq\frac{1}{3}\sum_{i=1}^{l}\left\|\sup\limits_{t\in[0,T]}e^{\rho t}|\Delta p_{0}(t,i)|^{2}\right\|_{\infty}.
\end{align*}
Hence, by a standard contraction mapping method, the following BSDE
\begin{equation}\label{overlineP0}
\left\{
\begin{array}{l}
d\widetilde{P}_{0}(t,i)=-[\Pi(t,i,\widetilde{P}_{0},\widetilde{\Lambda}_{0})+\widetilde{Q}(t,i)+\textstyle\sum_{j\neq i}q_{ij}\widetilde{P}_{0}(t,j)]dt+\widetilde{\Lambda}_{0}(t,i)dW(t),~t\in[0,T],\\
\widetilde{P}_{0}(T,i)=\widetilde{G}(i),~i\in\mathcal{M}.
\end{array}
\right.
\end{equation}
admit a unique solution $(\widetilde{P}_{0}(i),\widetilde{\Lambda}_{0}(i))_{i=1}^{l}$ such that $(\widetilde{P}_{0}(i),\widetilde{\Lambda}_{0}(i))\in L_{\mathcal{F}^{W}}^{\infty}(0,T;\mathbb{S}^{n})\times L_{\mathcal{F}^{W}}^{2,\mathrm{BMO}}(0,T;\mathbb{S}^{n})$ and $\tilde{P}(i)\geq 0$ for all $i\in\mathcal{M}$. Therefore, noting the following transform for all $i\in\mathcal{M}$ and $t\in[0,T]$,  
$$
P_{0}(t,i)=e^{-q_{ii}t}\widetilde{P}_{0}(t,i),\quad \Lambda_{0}(t,i)=e^{-q_{ii}t}\widetilde{\Lambda}_{0}(t,i),
$$
the proof is complete.
\end{proof}
\begin{remark}
We remark that the condition that $Q(t,i),G(i)\geq0$ for $i\in\mathcal{M},t\in[0,T]$ is only used to get $P_{0}(i)\geq0$ for $i\in\mathcal{M}$ in Lemma \ref*{le-b}. Moreover, we will get $P_{0}(i)\leq0$ for $i\in\mathcal{M}$ if $Q(t,i),G(i)\leq0$ for $i\in\mathcal{M},t\in[0,T]$.
\end{remark}

Even though the second part in Lemma \ref{le-b} can be regarded as a comparison result for matrix-valued BSDEs, we can not use it to show the monotonicity for our constructed sequence. This is because the regularity of the coefficients needed in Lemma \ref{le-b} is too strong to be satisfied in our situation. Therefore, we introduce another auxiliary BSDE and provide a comparison result in Lemma \ref{le-c}. Our result is inspired by and a slight extention of \cite[Lemma 5.2]{Peng 1992}.

We consider the following stochastic Riccati equation:
\begin{equation}\label{comparison}
\left\{
\begin{array}{l}
\begin{aligned}dM(t)=-[&M(t)\mathcal{A}(t)+\mathcal{A}(t)^{\top}M(t)+\mathcal{C}(t)^{\top}M(t)\mathcal{C}(t)+N(t)\mathcal{C}(t)\\
&+\mathcal{C}(t)^{\top}N(t)+\mathcal{Q}(t)]dt+N(t)dW(t),~t\in[0,T],
\end{aligned}\\
M(T)=\mathcal{G},~i\in\mathcal{M},
\end{array}
\right.
\end{equation}
with
$$
\begin{aligned}
&\mathcal{A}(t,\omega)\in L_{\mathcal{F}^{W}}^{2,\mathrm{BMO}}(0,T;\mathbb{R}^{n\times n}),\quad \mathcal{C}(t)=\mathcal{C}_{1}(t)+\mathcal{C}_{2}(t)\mathcal{C}_{3}(t),\\
&\mathcal{C}_{1}(t,\omega)\in L_{\mathcal{F}^{W}}^{\infty}(0,T;\mathbb{R}^{n\times n}),\quad,\mathcal{C}_{2}(t,\omega)\in L_{\mathcal{F}^{W}}^{\infty}(0,T;\mathbb{S}^{n}),\\
&\mathcal{C}_{3}(t,\omega)\in L_{\mathcal{F}^{W}}^{2,\mathrm{BMO}}(0,T;\mathbb{S}^{n}),\quad \mathcal{G}(\omega)\in L_{\mathcal{F}^{W}}^{\infty}(\Omega;\mathbb{S}^{n}),
\end{aligned}
$$
and $\mathcal{Q}(t,\omega)\in\mathcal{H}_{\mathcal{F}}^{p}(0,T;\mathbb{S}^{n})$ for any $p>2$.
\begin{lemma}\label{le-c}
Assume $\mathcal{G}\geq0,~\mathcal{Q}(t)\geq0$ for all $t\in[0,T]$ and BSDE \eqref{comparison} has a solution $(M,N)\in L_{\mathcal{F}^{W}}^{\infty}(0,T;\mathbb{S}^{n})\times L_{\mathcal{F}^{W}}^{2,\mathrm{BMO}}(0,T;\mathbb{S}^{n})$.
Then there exists a sufficiently small constant $\sigma_0>0$ (which depends on $\mathcal{C}_{3}$), when
$$
\sigma:=\left\|\mathcal{C}_{2}\right\|_{L_{\mathcal{F}^{W}}^{\infty}(0,T;\mathbb{S}^{n})}\leq\sigma_0,
$$
it holds that $M(t)\geq0$ for $t\in[0,T]$.
\end{lemma}
\begin{proof}
For $x\in\mathbb{R}^n$ and $t\in[0,T]$, we consider the following SDE
\begin{equation}\label{A,C}
\left\{\begin{array}{l}
d\mathcal{X}(s)=\mathcal{A}(s)\mathcal{X}(s)dt+\mathcal{C}(s)\mathcal{X}(s)dW(s),~s\in[t,T],\\
\mathcal{X}(t)=x.
\end{array}\right.
\end{equation}
which has a unique solution (see \cite{Gal 1979}). A direct computation yields that
\begin{equation}\label{SDE-1}
\left\{\begin{array}{l}
d|\mathcal{X}(s)|^{2}=2tr\left(\mathcal{X}(s)\mathcal{X}(s)^{\top}[\mathcal{A}(s)+\mathcal{C}(s)^{\top}\mathcal{C}(s)]\right)dt+2tr\left(\mathcal{X}(s)\mathcal{X}(s)^{\top}\mathcal{C}(s)\right)dW(s),~s\in[t,T],\\
|\mathcal{X}(t)|^{2}=|x|^{2}.
\end{array}\right.
\end{equation}
By Lemma \ref{le-a}, SDE \eqref{SDE-1} can be rewritten as
\begin{equation}
\left\{\begin{array}{l}
d|\mathcal{X}(s)|^{2}=|\mathcal{X}(s)|^{2}\left(\rho(s)\sigma^{2}|\mathcal{C}_{3}(s)|^{2}+\beta(s)|\mathcal{C}_{3}(s)|+\gamma(s)\right)dt+|\mathcal{X}(s)|^{2}\left(\delta(s)\sigma|\mathcal{C}_{3}(s)|+\eta(s)\right)dW(s),\\
\qquad\qquad\qquad\qquad\qquad\qquad\qquad\qquad\qquad\qquad\qquad\qquad\qquad\qquad\qquad\qquad~s\in[t,T],\\
|\mathcal{X}(t)|^{2}=|x|^{2}.
\end{array}\right.
\end{equation}
where $\rho(s),\beta(s),\gamma(s),\delta(s),\eta(s)$ are uniformly bounded and adapted processes. 

Then we can obtain that
\begin{align*}
|\mathcal{X}(s)|^{2}&=|x|^2\exp\left\{\int_{t}^{s}\left[\rho(r)\sigma^{2}|\mathcal{C}_{3}(r)|^{2}+\beta(r)|\mathcal{C}_{3}(r)|+\gamma(r)-\frac{1}{2}\left(\delta(r)\sigma|\mathcal{C}_{3}(r)|+\eta(r)\right)^{2}\right]dr\right.\\
&\left.\qquad+\int_{t}^{s}\left(\delta(r)\sigma|\mathcal{C}_{3}(r)|+\eta(r)\right)dW(r)\right\}.
\end{align*}
Take $q_2>2$, and let $\sigma_0>0$ be sufficiently small such that 
\begin{equation*}
\begin{aligned}
&\|\delta\|_{L_{\mathcal{F}^{W}}^{\infty}(0,T;\mathbb{R})}\sigma_0\|\mathcal{C}_{3}\|_{L_{\mathcal{F}^{W}}^{2,\mathrm{BMO}}(0,T;\mathbb{S}^{n})}<\Phi(q_{2}),\\
&\sqrt{q_{2}\|\rho\|_{L_{\mathcal{F}^{W}}^{\infty}(0,T;\mathbb{R})}}\sigma_0\|\mathcal{C}_{3}\|_{L_{\mathcal{F}^{W}}^{2,\mathrm{BMO}}(0,T;\mathbb{S}^{n})}<1,
\end{aligned}
\end{equation*}
where
$$
\Phi(x)=\sqrt{1+\frac{1}{x^{2}}\log(\frac{2x-1}{2(x-1)})}-1,~1<x<\infty.
$$
When $\sigma\leq \sigma_0$, by H\"older's inequality, it holds that for $2<q<q_{1}<q_{2}$,
\begin{align*}
&\mathbb{E}\left[\sup_{t\leq s\leq T}|\mathcal{X}(t)|^{q}\right]\\
=&|x|^{q}\mathbb{E}\left[\sup_{t\leq s\leq T}\exp\left\{\int_{t}^{s}\left[\frac{q}{2}\gamma(r)-\frac{q}{4}\eta(r)^{2}\right]dr+\int_{t}^{s}\frac{q}{2}\eta(r)dW(r)\right\}\cdot\exp\left\{\int_{t}^{s}\frac{q}{2}\delta(r)\sigma|\mathcal{C}_{3}(r)|dW(r)\right\}\right.\\
&\cdot\left.\exp\left\{\int_{t}^{s}\left[\left(\frac{q}{2}\rho(r)-\frac{q}{4}\delta(r)^{2}\right)\sigma^{2}|\mathcal{C}_{3}(r)|^{2}+\frac{q}{2}\left(\beta(r)-\delta(r)\eta(r)\sigma\right)|\mathcal{C}_{3}(r)|\right]dr\right\}\right]\\
\leq&N\mathbb{E}\left[\sup_{t\leq s\leq T}\exp\left\{\int_{t}^{s}\frac{q_{1}}{2}\delta(r)\sigma|\mathcal{C}_{3}(r)|dW(r)\right\}\right.\\
&\cdot\left.\exp\left\{\int_{t}^{s}\left[\left(\frac{q_{1}}{2}\rho(r)-\frac{q_{1}}{4}\delta(r)^{2}\right)\sigma^{2}|\mathcal{C}_{3}(r)|^{2}+\frac{q_{1}}{2}\left(\beta(r)-\delta(r)\eta(r)\sigma\right)|\mathcal{C}_{3}(r)|\right]dr\right\}\right]^{\frac{q}{q_1}}\\
\leq&N\mathbb{E}\left[\sup_{t\leq s\leq T}\left\{\mathcal{E}_{t,s}\left(\delta\sigma|\mathcal{C}_{3}|\cdot W\right)\right\}^{\frac{q_{1}}{2}}\cdot\exp\left\{\int_{t}^{s}\frac{q_{1}}{2}\rho(r)\sigma^{2}|\mathcal{C}_{3}(r)|^{2}dr\right\}\right.\\
&\cdot\left.\exp\left\{\int_{t}^{s}\left[\frac{q_{1}\kappa^{2}}{2}|\mathcal{C}_{3}(r)|^{2}+\frac{q_{1}\left(\beta(r)-\delta(r)\eta(r)\sigma\right)^{2}}{8\kappa^{2}}\right]dr\right\}\right]^{\frac{q}{q_1}}\\
\leq&N\mathbb{E}\left[\sup_{t\leq s\leq T}\left\{\mathcal{E}_{t,s}\left(\delta\sigma|\mathcal{C}_{3}|\cdot W\right)\right\}^{\frac{q_{2}}{2}}\cdot\exp\left\{\int_{t}^{s}\frac{q_{2}}{2}(\rho(r)\sigma^{2}+\kappa^2)|\mathcal{C}_{3}(r)|^{2}dr\right\}\right]^{\frac{q}{q_2}}\\
\leq&N\mathbb{E}\left[\sup_{t\leq s\leq T}\left\{\mathcal{E}_{t,s}\left(\delta(i)\sigma|\mathcal{C}_{3}|\cdot W\right)\right\}^{q_{2}}\right]^{\frac{q}{2q_2}}\mathbb{E}\left[\sup_{t\leq s\leq T}\exp\left\{\int_{t}^{s}q_{2}(\rho(r)\sigma^{2}+\kappa^2)|\mathcal{C}_{3}(r)|^{2}dr\right\}\right]^{\frac{q}{2q_2}}\\
\leq&N,
\end{align*}
where $\kappa$ is a sufficiently small positive constant, $N$ is positive constant changing from line to line and the last inequality follows from \cite[Theorem 2.2 and Theorem 3.1]{Kazamaki 1994}. Finally by It\^o's formula, it holds that 
\begin{align*}
\mathbb{E}_{t}\left[\int_{t}^{T}d\langle M(s)\mathcal{X}(s),\mathcal{X}(s)\rangle\right]=&\mathbb{E}_{t}\left[\int_{t}^{T}\langle dM(s)\mathcal{X}(s),\mathcal{X}(s)\rangle\right]+\mathbb{E}_{t}\left[\int_{t}^{T}\langle M(s)\mathcal{X}(s),d\mathcal{X}(s)\rangle\right]\\
&+\mathbb{E}_{t}\left[\int_{t}^{T}\langle dM(s)\mathcal{X}(s),d\mathcal{X}(s)\rangle\right]\\
=&-\mathbb{E}_{t}\left[\int_{t}^{T}\langle \mathcal{Q}(s)\mathcal{X}(s),\mathcal{X}(s)\rangle ds\right].
\end{align*}
Therefore we have
\begin{align*}
\langle M(t)x,x\rangle=\mathbb{E}_{t}\left[\int_{t}^{T}\langle \mathcal{Q}(s)\mathcal{X}(s),\mathcal{X}(s)\rangle ds+\langle \mathcal{G}\mathcal{X}(T),\mathcal{X}(T)\rangle\right],
\end{align*}
which completes the proof.
\end{proof}

\begin{remark}
Clearly, we can get $M(t)\leq0$ if we replace $\mathcal{G}\geq0,~\mathcal{Q}(t)\geq0$ by $\mathcal{G}\leq0,~\mathcal{Q}(t)\leq0$ for all $t\in[0,T]$ in Lemma \ref{le-c}.
\end{remark}

\subsection{Existence of solution for ESRE \eqref{Riccati}}
We are now ready to prove the existence of solution for ESRE \eqref{Riccati}. We first construct a monotone sequence $\left\{(\widetilde{P}_{k},\widetilde{\Lambda}_{k})\right\}_{k\geq0}$ and provide a priori estimates. Next we obtain the convergence of $\left\{\widetilde{P}_{k}\right\}_{k\geq0}$ by monotone convergence theorem and dominated convergence theorem. In order to prove the convergence of $\left\{\widetilde{\Lambda}_{k}\right\}_{k\geq0}$, we need to show that $\left\{\widetilde{P}_{k}\right\}_{k\geq0}$ converges in some finer space. The key ingredient is to link $\left\{(\widetilde{P}_{k},\widetilde{\Lambda}_{k})\right\}_{k\geq1}$ with the solutions of fully coupled FBSDEs $\left\{(\mathbf{X}_k,\mathbf{Y}_k,\mathbf{Z}_k)\right\}_{k\geq 1}$. We establish $L^p$ estimates and obtain convergence of $\left\{(\mathbf{X}_k,\mathbf{Y}_k,\mathbf{Z}_k)\right\}_{k\geq 1}$. Finally we get the convergence of $\left\{(\widetilde{P}_{k},\widetilde{\Lambda}_{k})\right\}_{k\geq0}$ and obtain a solution for ESRE \eqref{Riccati}.

\begin{proof}[Proof\ of\ Theorem \ref{th-a}]The proof is divided into several steps.\\
$\textbf{Step 1. (Construction of Piccard iteration schemes)}$ To begin with, let $(\widetilde{P}_{0}(i),\widetilde{\Lambda}_{0}(i))_{i=1}^{l}$ be the unique solution of \eqref{overlineP0} such that $(\widetilde{P}_{0}(i),\widetilde{\Lambda}_{0}(i))\in L_{\mathcal{F}^{W}}^{\infty}(0,T;\mathbb{S}^{n})\times L_{\mathcal{F}^{W}}^{2,\mathrm{BMO}}(0,T;\mathbb{S}^{n})$ and $\tilde{P}(i)\geq 0$ for all $i\in\mathcal{M}$. Then it follows from \cite[Theorem 5.1]{Peng 1992} and \cite[Theorem 2.4]{Tang 2003} that 
\begin{equation*}
\left\{
\begin{array}{l}
\begin{aligned}
d\widetilde{P}_{1}(t,i)=&-[\Pi(t,i,\widetilde{P}_{1},\widetilde{\Lambda}_{1})+\widetilde{Q}(t,i)+H(t,i,\widetilde{P}_{1},\widetilde{\Lambda}_{1},\widetilde{R},\widetilde{S})\\
&\quad+\sum_{j\neq i}q_{ij}e^{(q_{ii}-q_{jj})t}\widetilde{P}_{0}(t,j)]dt+\widetilde{\Lambda}_{1}(t,i)dW(t),
\end{aligned}\\
\widetilde{R}(t,i)+D(t,i)^{\top}\widetilde{P}_{1}(t,i)D(t,i)>0,\\
\widetilde{P}_{1}(T,i)=\widetilde{G}(i),~t\in[0,T],~i\in\mathcal{M},
\end{array}
\right.
\end{equation*}
with
\begin{equation*}
\begin{aligned}
&\widetilde{Q}(t,i)=e^{q_{ii}t}Q(t,i),\quad\widetilde{R}(t,i)=e^{q_{ii}t}R(t,i),\\
&\widetilde{S}(t,i)=e^{q_{ii}t}S(t,i),\quad\widetilde{G}(i)=e^{q_{ii}T}G(i),
\end{aligned}
\end{equation*}
admits a unique solution $(\widetilde{P}_{1}(i),\widetilde{\Lambda}_{1}(i))_{i=1}^{l}$ such that $(\widetilde{P}_{1}(i),\widetilde{\Lambda}_{1}(i))\in L_{\mathcal{F}^{W}}^{\infty}(0,T;\mathbb{S}^{n})\times L_{\mathcal{F}^{W}}^{2}(0,T;\mathbb{S}^{n})$ and $\widetilde{P}(i)\geq 0$ for all $i\in\mathcal{M}$. On the other hand, it is easy to verify that $H(t,i,\widetilde{P}_{1},\widetilde{\Lambda}_{1},\widetilde{R},\widetilde{S})\leq 0$. Therefore, it follows from a similar arugment as in Lemma \ref{le-estimate} that $\widetilde{\Lambda}_{1}(i)\in L_{\mathcal{F}^{W}}^{2,\mathrm{BMO}}(0,T;\mathbb{S}^{n})$. By induction we obtain a family of $\left\{(\widetilde{P}_{k}(i),\widetilde{\Lambda}_{k}(i))_{i=1}^{l}\right\}_{k\in\mathbb{N}}$ such that for all $k\in\mathbb{N}$, $(\widetilde{P}_{k}(i),\widetilde{\Lambda}_{k}(i))\in L_{\mathcal{F}^{W}}^{\infty}(0,T;\mathbb{S}^{n})\times L_{\mathcal{F}^{W}}^{2,\mathrm{BMO}}(0,T;\mathbb{S}^{n})$ satisfying $\widetilde{P}_{k}(i)\geq0$ for $i\in\mathcal{M}$ and
\begin{equation}\label{overlinePk}
\left\{
\begin{array}{l}
\begin{aligned}
d\widetilde{P}_{k+1}(t,i)=&-[\Pi(t,i,\widetilde{P}_{k+1},\widetilde{\Lambda}_{k+1})+\widetilde{Q}(t,i)+H(t,i,\widetilde{P}_{k+1},\widetilde{\Lambda}_{k+1},\widetilde{R},\widetilde{S})\\
&\quad+\sum_{j\neq i}q_{ij}e^{(q_{ii}-q_{jj})t}\widetilde{P}_{k}(t,j)]dt+\widetilde{\Lambda}_{k+1}(t,i)dW(t),
\end{aligned}\\
\widetilde{R}(t,i)+D(t,i)^{\top}\widetilde{P}_{k+1}(t,i)D(t,i)>0,\\
\widetilde{P}_{k+1}(T,i)=\widetilde{G}(i),~t\in[0,T],~i\in\mathcal{M}.
\end{array}
\right.
\end{equation}
Moreover, noting the non-negativity of $\tilde{P}_{k}$ and hence the non-positivity of $H(t,i,\widetilde{P}_{k},\widetilde{\Lambda}_{k},\widetilde{R},\widetilde{S})$, similar to Lemma \ref{le-estimate}, we have for all $k\in\mathbb{N}$ and $i\in\mathcal{M}$,
\begin{equation}\label{BMO2}
\begin{aligned}
&\|e^{\frac{1}{2}\rho\cdot}\widetilde{P}_{k}(\cdot,i)\|_{L_{\mathcal{F}^{W}}^{\infty}(0,T;\mathbb{S}^{n})}\leq\sqrt{\frac{3}{2}e^{\rho T}(K^{2}+\frac{1}{\rho})},\\
&\|e^{\frac{1}{2}\rho\cdot}\widetilde{\Lambda}_{k}(\cdot,i)\|_{L_{\mathcal{F}^{W}}^{2,\mathrm{BMO}}(0,T;\mathbb{S}^{n})}\leq \sqrt{3e^{\rho T}(K^{2}+\frac{1}{\rho})},
\end{aligned}
\end{equation}
where $\rho=\left(3(l-1)^{2}T+3\right)K^{2}+3K$.\\
\\
$\textbf{Step 2. (Monotone convergence of $\widetilde{P}_{k}$)}$
From the construction of $\widetilde{P}_{0}(i)$ and $\widetilde{P}_{1}(i)$, we obtain that
\begin{align*}
d\left(\widetilde{P}_{1}(t,i)-\widetilde{P}_{0}(t,i)\right)=&-\left[\Pi(t,i,\widetilde{P}_{1}-\widetilde{P}_{0},\widetilde{\Lambda}_{1}-\widetilde{\Lambda}_{0})+H(t,i,\widetilde{P}_{1},\widetilde{\Lambda}_{1},\widetilde{R},\widetilde{S})\right]dt\\
&+\left(\widetilde{\Lambda}_{1}(t,i)-\widetilde{\Lambda}_{0}(t,i)\right)dW(t).
\end{align*}
Since
\begin{equation*}
\widetilde{P}_{1}(t,i)\geq0,~\widetilde{R}(t,i)>0,~i\in\mathcal{M},~t\in[0,T],
\end{equation*}
then
\begin{equation*}
H(t,i,\widetilde{P}_{1},\widetilde{\Lambda}_{1},\widetilde{R},\widetilde{S})\leq0,~i\in\mathcal{M},~t\in[0,T].
\end{equation*}
Also with $\widetilde{P}_{1}(T,i)-\widetilde{P}_{0}(T,i)=0$, we obtain from \cite[Lemma 5.2]{Peng 1992} that
\begin{equation*}
\widetilde{P}_{1}(t,i)\leq\widetilde{P}_{0}(t,i),~i\in\mathcal{M},~t\in[0,T].
\end{equation*}
To prove
\begin{equation*}
\widetilde{P}_{k+1}(t,i)\leq\widetilde{P}_{k}(t,i),~k\in\mathbb{N},~i\in\mathcal{M},~t\in[0,T],
\end{equation*}
we first define
\begin{equation}\label{GF}
\left\{
\begin{array}{l}
\begin{aligned}
G(t,i,\widetilde{P},\widetilde{\Lambda},\theta)&=\left(A(t,i)+B(t,i)\theta\right)^{\top}\widetilde{P}(t,i)+\widetilde{P}(t,i)\left(A(t,i)+B(t,i)\theta\right)\\
&\quad+\left(C(t,i)+D(t,i)\theta\right)^{\top}\widetilde{\Lambda}(t,i)+\widetilde{\Lambda}(t,i)\left(C(t,i)+D(t,i)\theta\right)\\
&\quad+\left(C(t,i)+D(t,i)\theta\right)^{\top}\widetilde{P}(t,i)\left(C(t,i)+D(t,i)\theta\right),
\end{aligned}\\
F(t,i,\widetilde{P},\widetilde{\Lambda},\theta)=G(t,i,\widetilde{P},\widetilde{\Lambda},\theta)+\theta^{\top}\widetilde{S}(t,i)+\widetilde{S}(t,i)^{\top}\theta+\theta^{\top}\widetilde{R}(t,i)\theta+\widetilde{Q}(t,i),
\end{array}
\right.
\end{equation}
for $\theta\in\mathbb{R}^{m\times n}$ and for $\widetilde{R}(t,i)+D(t,i)^{\top}\widetilde{P}(t,i)D(t,i)>0$
\begin{equation}\label{theta}
\begin{aligned}
\widehat{\theta}(t,i,\widetilde{P},\widetilde{\Lambda})&=-\left(\widetilde{R}(t,i)+D(t,i)^{\top}\widetilde{P}(t,i)D(t,i)\right)^{-1}\\
&\quad\cdot\left(B(t,i)^{\top}\widetilde{P}(t,i)+D(t,i)^{\top}\widetilde{P}(t,i)C(t,i)+D(t,i)^{\top}\widetilde{\Lambda}(t,i)+\widetilde{S}(t,i)\right).
\end{aligned}\\
\end{equation}
Since
\begin{align*}
\left.\frac{\partial F(t,i,\widetilde{P},\widetilde{\Lambda},\theta)}{\partial\theta}\right|_{\theta=\widehat{\theta}(t,i,\widetilde{P},\widetilde{\Lambda})}=0,
\end{align*}
then
\begin{align*}
F\left(t,i,\widetilde{P},\widetilde{\Lambda},\widehat{\theta}(t,i,\widetilde{P},\widetilde{\Lambda})\right)\leq F\left(t,i,\widetilde{P},\widetilde{\Lambda},\theta\right),~i\in\mathcal{M},~t\in[0,T].
\end{align*}
Now we consider
\begin{align*}
d\left(\widetilde{P}_{2}(t,i)-\widetilde{P}_{1}(t,i)\right)=&-\left[F\left(t,i,\widetilde{P}_{2},\widetilde{\Lambda}_{2},\widehat{\theta}(t,i,\widetilde{P}_{2},\widetilde{\Lambda}_{2})\right)
-F\left(t,i,\widetilde{P}_{1},\widetilde{\Lambda}_{1},\widehat{\theta}(t,i,\widetilde{P}_{1},\widetilde{\Lambda}_{1})\right)\right.\\
&\quad+\textstyle\sum_{j\neq i}\left.q_{ij}e^{(q_{ii}-q_{jj})t}\left({P}_{1}(t,j)-{P}_{0}(t,j)\right)\right]dt+\left(\widetilde{\Lambda}_{2}(t,i)-\widetilde{\Lambda}_{1}(t,i)\right)dW(t)\\
=&-\left[F\left(t,i,\widetilde{P}_{2},\widetilde{\Lambda}_{2},\widehat{\theta}(t,i,\widetilde{P}_{2},\widetilde{\Lambda}_{2})\right)
-F\left(t,i,\widetilde{P}_{2},\widetilde{\Lambda}_{2},\widehat{\theta}(t,i,\widetilde{P}_{1},\widetilde{\Lambda}_{1})\right)\right.\\
&\quad+F\left(t,i,\widetilde{P}_{2},\widetilde{\Lambda}_{2},\widehat{\theta}(t,i,\widetilde{P}_{1},\widetilde{\Lambda}_{1})\right)
-F\left(t,i,\widetilde{P}_{1},\widetilde{\Lambda}_{1},\widehat{\theta}(t,i,\widetilde{P}_{1},\widetilde{\Lambda}_{1})\right)\\
&\quad+\textstyle\sum_{j\neq i}\left.q_{ij}e^{(q_{ii}-q_{jj})t}\left({P}_{1}(t,j)-{P}_{0}(t,j)\right)\right]dt+\left(\widetilde{\Lambda}_{2}(t,i)-\widetilde{\Lambda}_{1}(t,i)\right)dW(t)\\
=&-\left[F\left(t,i,\widetilde{P}_{2},\widetilde{\Lambda}_{2},\widehat{\theta}(t,i,\widetilde{P}_{2},\widetilde{\Lambda}_{2})\right)
-F\left(t,i,\widetilde{P}_{2},\widetilde{\Lambda}_{2},\widehat{\theta}(t,i,\widetilde{P}_{1},\widetilde{\Lambda}_{1})\right)\right.\\
&\quad+G\left(t,i,\widetilde{P}_{2}-\widetilde{P}_{1},\widetilde{\Lambda}_{2}-\widetilde{\Lambda}_{1},\widehat{\theta}(t,i,\widetilde{P}_{1},\widetilde{\Lambda}_{1})\right)\\
&\quad+\textstyle\sum_{j\neq i}\left.q_{ij}e^{(q_{ii}-q_{jj})t}\left({P}_{1}(t,j)-{P}_{0}(t,j)\right)\right]dt+\left(\widetilde{\Lambda}_{2}(t,i)-\widetilde{\Lambda}_{1}(t,i)\right)dW(t).
\end{align*}
Since
\begin{align*}
&F\left(t,i,\widetilde{P}_{2},\widetilde{\Lambda}_{2},\widehat{\theta}(t,i,\widetilde{P}_{2},\widetilde{\Lambda}_{2})\right)-F\left(t,i,\widetilde{P}_{2},\widetilde{\Lambda}_{2},\widehat{\theta}(t,i,\widetilde{P}_{1},\widetilde{\Lambda}_{1})\right)\leq0,\\
&q_{ij}\geq0,~\widetilde{P}_{1}(t,j)-\widetilde{P}_{0}(t,j)\leq0,~\widetilde{P}_{2}(T,i)-\widetilde{P}_{1}(T,i)=0,\quad i\in\mathcal{M},~j\neq i,~t\in[0,T],
\end{align*}
it follows from Lemma \ref{le-c} that there exists a sufficiently small constant $\sigma_0$ such that when 
\begin{equation*}
  e^{-q_{ii}t}|D(t,i)R(t,i)^{-1}D(t,i)^{\top}|\leq \sigma_0,~~i\in\mathcal{M},~t\in[0,T],
\end{equation*}
it holds that
\begin{align*}
\widetilde{P}_{2}(t,i)\leq\widetilde{P}_{1}(t,i),~i\in\mathcal{M},~t\in[0,T].
\end{align*}
Similarly, by induction we have
\begin{align*}
\widetilde{P}_{k+1}(t,i)\leq\widetilde{P}_{k}(t,i),\quad k\in\mathbb{N},~i\in\mathcal{M},~t\in[0,T].
\end{align*}
To conclude, it holds that
\begin{align*}
0\leq\widetilde{P}_{k+1}(t,i)\leq\widetilde{P}_{k}(t,i)\leq\widetilde{P}_{0}(t,i),~k\in\mathbb{N},~i\in\mathcal{M},~t\in[0,T].
\end{align*}
Therefore, by monotone convergence theorem and dominated convergence theorem, for each $i\in\mathcal{M}$, there exists an adapted process $\widetilde{P}(i)$ such that for any $p>2$,
\begin{align*}
\lim_{k\rightarrow\infty}\mathbb{E}\left[\int_{0}^{T}\left|\widetilde{P}_{k}(t,i)-\widetilde{P}(t,i)\right|^{p}ds\right]=0.
\end{align*}
Without loss of any generality, we also assume that for all $t\in[0,T]$ and $p>2$,
\begin{align*}
\lim_{k\rightarrow\infty}\mathbb{E}\left[\left|\widetilde{P}_{k}(t,i)-\widetilde{P}(t,i)\right|^{p}\right]=0.
\end{align*}
\\
\textbf{Step 3. ($L^{p}$ estimates and convergence of $\left\{(\mathbf{X}_k,\mathbf{Y}_k,\mathbf{Z}_k)\right\}_{k\geq 1}$)} Inspired by the connection between stochastic Riccati equations and FBSDEs, we consider FBSDEs of the following types:
\begin{equation}\label{FBSDE-1}
\left\{
\begin{array}{l}
d\mathbf{X}_{k}(t,i)=\left[A(t,i)\mathbf{X}_{k}(t,i)+B(t,i)\mathbf{u}_{k}(t,i)\right]dt+\left[C(t,i)\mathbf{X}_{k}(t,i)+D(t,i)\mathbf{u}_{k}(t,i)\right]dW(t),\\
\begin{aligned}
d\mathbf{Y}_{k}(t,i)=&-\left[A(t,i)^{\top}\mathbf{Y}_{k}(t,i)\right.+C(t,i)^{\top}\mathbf{Z}_{k}(t,i)+\left(\widetilde{Q}(t,i)+\textstyle\sum_{j\neq i}q_{ij}e^{(q_{ii}-q_{jj})t}\widetilde{P}_{k-1}(t,j)\right)\mathbf{X}_{k}(t,i)\\
&+\left.\widetilde{S}(t,i)^{\top}\mathbf{u}_{k}(t,i)\right]dt+\mathbf{Z}_{k}(t,i)dW(t),
\end{aligned}\\
\mathbf{X}_{k}(0,i)=\mathbf{I}_{n},~\mathbf{Y}_{k}(T,i)=\widetilde{G}(i)\mathbf{X}_{k}(T,i),\\
\mathbf{u}_{k}(t,i)=-\widetilde{R}(t,i)^{-1}\left(B(t,i)^{\top}\mathbf{Y}_{k}(t,i)+D(t,i)^{\top}\mathbf{Z}_{k}(t,i)+\widetilde{S}(t,i)\mathbf{X}_{k}(t,i)\right),~t\in[0,T],~ k\geq1,~i\in\mathcal{M},
\end{array}
\right.
\end{equation}
in which the following relationship holds on $[0,T]$:
\begin{equation}\label{Y=PX}
\begin{aligned}
&\mathbf{Y}_{k}(t,i)=\widetilde{P}_{k}(t,i)\mathbf{X}_{k}(t,i),\\
&\mathbf{Z}_{k}(t,i)=\widetilde{\Lambda}_{k}(t,i)\mathbf{X}_{k}(t,i)+\widetilde{P}_{k}(t,i)C(t,i)\mathbf{X}_{k}(t,i)+\widetilde{P}_{k}(t,i)D(t,i)\mathbf{u}_{k}(t,i).
\end{aligned}
\end{equation}
More details for the  relationship \eqref{Y=PX} can be found in \cite{Sun et al. 2021,Tang 2003}. We rewrite \eqref{FBSDE-1} as
\begin{equation}\label{FBSDE-2}
\left\{
\begin{array}{l}
\begin{aligned}
d\mathbf{X}_{k}(t,i)=&\left[\left(A(t,i)-B(t,i)\widetilde{R}(t,i)^{-1}\widetilde{S}(t,i)\right)\mathbf{X}_{k}(t,i)-B(t,i)\widetilde{R}(t,i)^{-1}B(t,i)^{\top}\mathbf{Y}_{k}(t,i)\right.\\
&\left.-B(t,i)\widetilde{R}(t,i)^{-1}D(t,i)^{\top}\mathbf{Z}_{k}(t,i)\right]dt+\left[\left(C(t,i)-D(t,i)\widetilde{R}(t,i)^{-1}\widetilde{S}(t,i)\right)\mathbf{X}_{k}(t,i)\right.\\
&\left.-D(t,i)\widetilde{R}(t,i)^{-1}B(t,i)^{\top}\mathbf{Y}_{k}(t,i)-D(t,i)\widetilde{R}(t,i)^{-1}D(t,i)^{\top}\mathbf{Z}_{k}(t,i)\right]dW(t)\\
\triangleq&b(t,i,\mathbf{X}_{k}(t,i),\mathbf{Y}_{k}(t,i),\mathbf{Z}_{k}(t,i))dt+\sigma(t,i,\mathbf{X}_{k}(t,i),\mathbf{Y}_{k}(t,i),\mathbf{Z}_{k}(t,i))dW(t),
\end{aligned}\\
\begin{aligned}
d\mathbf{Y}_{k}(t,i)=&-\left[\left(\widetilde{Q}(t,i)-\widetilde{S}(t,i)^{\top}\widetilde{R}(t,i)^{-1}\widetilde{S}(t,i)+\textstyle\sum_{j\neq i}q_{ij}e^{(q_{ii}-q_{jj})t}\widetilde{P}_{k-1}(t,j)\right)\mathbf{X}_{k}(t,i)\right.\\
&\quad+\left(A(t,i)^{\top}-\widetilde{S}(t,i)^{\top}\widetilde{R}(t,i)^{-1}B(t,i)^{\top}\right)\mathbf{Y}_{k}(t,i)\\
&\quad+\left.\left(C(t,i)^{\top}-\widetilde{S}(t,i)^{\top}\widetilde{R}(t,i)^{-1}D(t,i)^{\top}\right)\mathbf{Z}_{k}(t,i)\right]dt+\mathbf{Z}_{k}(t,i)dW(t)\\
\triangleq&-f_{k}(t,i,\mathbf{X}_{k}(t,i),\mathbf{Y}_{k}(t,i),\mathbf{Z}_{k}(t,i))dt+\mathbf{Z}_{k}(t,i)dW(t),
\end{aligned}\\
\mathbf{X}_{k}(0,i)=\mathbf{I}_{n},~\mathbf{Y}_{k}(T,i)=\widetilde{G}(i)\mathbf{X}_{k}(T,i)\triangleq h(i,\mathbf{X}_{k}(T,i)),\qquad k\geq1,~i\in\mathcal{M},~t\in[0,T].
\end{array}
\right.
\end{equation}
Rcalling uniform boundedness of $\widetilde{P}_{k}(t,j)$, and assumption $(\mathscr{A}1)$, we denote $L$ the Lipschitz coefficient of $b,\sigma,f_k,h$ and in particular $L_{\sigma,z}$ the Lipschitz coefficient of $\sigma$ with respect to $z$. We will also assume that $L_{\sigma,z}\leq L$.

Since for $k\geq1,i\in\mathcal{M}$, $\mathbf{Y}_{k}(t,i)=\widetilde{P}_{k}(t,i)\mathbf{X}_{k}(t,i)$ holds on $[0,T]$ and $\widetilde{P}_k(t,i)$ is uniformly bounded, a pasting technique combing with $L^{p}$ estimate for FBSDEs on small time horizon (see \cite{Yong 2020}) allows us to obtain the following $L^{p}$ estimate for FBSDE \eqref{FBSDE-2}, i.e., for every $p>2$, there exists sufficiently small constant $\bar{L}_{\sigma,p}>0$ such that when $L_{\sigma,z}\leq \bar{L}_{\sigma,p}$, it holds that
\begin{align*}
\mathbb{E}\left[\sup\limits_{0\leq t\leq T}|\mathbf{X}_{k}(t,i)|^{p}+\sup\limits_{0\leq t\leq T}|\mathbf{Y}_{k}(t,i)|^{p}+\left(\int_{0}^{T}|\mathbf{Z}_{k}(t,i)|^{2}dt\right)^{\frac{p}{2}}\right]\leq B_{p}
\end{align*}
where $B_{p}$ is a positive constant only depending on $p,L$.

To prove the convergence of ${(\mathbf{X}_{k}(i),\mathbf{Y}_{k}(i),\mathbf{Z}_{k}(i))}$, we denote
\begin{align*}
&\Delta\mathbf{X}_{k+1}(t,i)=\mathbf{X}_{k+1}(t,i)-\mathbf{X}_{k}(t,i),~\Delta\mathbf{Y}_{k+1}(t,i)=\mathbf{Y}_{k+1}(t,i)-\mathbf{Y}_{k}(t,i),\\
&\Delta\mathbf{Z}_{k+1}(t,i)=\mathbf{Z}_{k+1}(t,i)-\mathbf{Z}_{k}(t,i),~\Delta\mathbf{u}_{k+1}(t,i)=\mathbf{u}_{k+1}(t,i)-\mathbf{u}_{k}(t,i),\\
&\Delta\widetilde{P}_{k+1}(t,i)=\widetilde{P}_{k+1}(t,i)-\widetilde{P}_{k}(t,i),\quad k\geq1,~i\in\mathcal{M},~t\in[0,T],
\end{align*}
and obtain that
\begin{equation}\label{FBSDE-3}
\left\{
\begin{array}{l}
\begin{aligned}
d\Delta\mathbf{X}_{k}(t,i)=&\left[\left(A(t,i)-B(t,i)\widetilde{R}(t,i)^{-1}\widetilde{S}(t,i)\right)\Delta\mathbf{X}_{k}(t,i)-B(t,i)\widetilde{R}(t,i)^{-1}B(t,i)^{\top}\Delta\mathbf{Y}_{k}(t,i)\right.\\
&\left.-B(t,i)\widetilde{R}(t,i)^{-1}D(t,i)^{\top}\Delta\mathbf{Z}_{k}(t,i)\right]dt+\left[\left(C(t,i)-D(t,i)\widetilde{R}(t,i)^{-1}\widetilde{S}(t,i)\right)\Delta\mathbf{X}_{k}(t,i)\right.\\
&\left.-D(t,i)\widetilde{R}(t,i)^{-1}B(t,i)^{\top}\Delta\mathbf{Y}_{k}(t,i)-D(t,i)\widetilde{R}(t,i)^{-1}D(t,i)^{\top}\Delta\mathbf{Z}_{k}(t,i)\right]dW(t).
\end{aligned}\\
\begin{aligned}
d\Delta\mathbf{Y}_{k}(t,i)=&-\left[\left(\widetilde{Q}(t,i)-\widetilde{S}(t,i)^{\top}\widetilde{R}(t,i)^{-1}\widetilde{S}(t,i)+\textstyle\sum_{j\neq i}q_{ij}e^{(q_{ii}-q_{jj})t}\widetilde{P}_{k-1}(t,j)\right)\Delta\mathbf{X}_{k}(t,i)\right.\\
&\quad+\left(A(t,i)^{\top}-\widetilde{S}(t,i)^{\top}\widetilde{R}(t,i)^{-1}B(t,i)^{\top}\right)\Delta\mathbf{Y}_{k}(t,i)\\
&\quad+\left(C(t,i)^{\top}-\widetilde{S}(t,i)^{\top}\widetilde{R}(t,i)^{-1}D(t,i)^{\top}\right)\Delta\mathbf{Z}_{k}(t,i)\\
&\quad+\textstyle\sum_{j\neq i}\left.q_{ij}e^{(q_{ii}-q_{jj})t}\Delta\widetilde{P}_{k-1}(t,j)\mathbf{X}_{k-1}(t,i)\right]dt+\Delta\mathbf{Z}_{k}(t,i)dW(t).
\end{aligned}\\
\Delta\mathbf{X}_{k}(0,i)=\mathbf{0},~\Delta\mathbf{Y}_{k}(T,i)=\widetilde{G}(i)\Delta\mathbf{X}_{k}(T,i),\quad k\geq2,~i\in\mathcal{M},~t\in[0,T].
\end{array}
\right.
\end{equation}
For $0<\varepsilon< T$ and $t\in[0,T-\varepsilon]$, standard technique for BSDEs implies that 
$$
\begin{aligned}
|\Delta\mathbf{Y}_{k}(t,i)|^{2}\leq&C\mathbb{E}_{t}\left[|\Delta\mathbf{Y}_{k}(t+\varepsilon,i)|^{2}\right]+C\mathbb{E}_{t}\left[\sup\limits_{t\leq s\leq t+\varepsilon}|\Delta\mathbf{X}_{k}(s,i)|^{2}\right]\\
&+CK^{2}\textstyle\sum_{j\neq i}\mathbb{E}_{t}\left[\int_{t}^{t+\varepsilon}|\Delta\widetilde{P}_{k-1}(s,j)\mathbf{X}_{k-1}(s,i)|^{2}ds\right]\\
\leq&C\mathbb{E}_{t}\left[|\Delta\widetilde{P}_{k}(t+\varepsilon,i)\mathbf{X}_{k}(t+\varepsilon,i)|^{2}\right]+C\mathbb{E}_{t}\left[\sup\limits_{t\leq s\leq t+\varepsilon}|\Delta\mathbf{X}_{k}(s,i)|^{2}\right]\\
&+CK^{2}\textstyle\sum_{j\neq i}\mathbb{E}_{t}\left[\int_{t}^{t+\varepsilon}|\Delta\widetilde{P}_{k-1}(s,j)\mathbf{X}_{k-1}(s,i)|^{2}ds\right],
\end{aligned}
$$
where $C$ only depends on Lipschitz constants $L$ and the uniform bound of $\widetilde{P}_{k-1}$, changing from line to line.

For $p>2$, it follows from Doob's martingale inequality that
\begin{equation}\label{Y}
\begin{aligned}
\mathbb{E}\left[\sup\limits_{t\leq s\leq t+\varepsilon}|\Delta\mathbf{Y}_{k}(s,i)|^{p}\right]\leq &C_{p}\mathbb{E}\left[|\Delta\widetilde{P}_{k}(t+\varepsilon,i)\mathbf{X}_{k}(t+\varepsilon,i)|^{2}\right]+C_{p}\mathbb{E}\left[\sup\limits_{t\leq s\leq t+\varepsilon}|\Delta\mathbf{X}_{k}(s,i)|^{p}\right]\\
&+C_{p}K^{p}\textstyle\sum_{j\neq i}\mathbb{E}\left[\int_{t}^{t+\varepsilon}|\Delta\widetilde{P}_{k-1}(s,j)\mathbf{X}_{k-1}(s,i)|^{p}ds\right],
\end{aligned}
\end{equation}
where $C_{p}$ depends on $p,L$ and the uniform bound of $\widetilde{P}_{k-1}$. From the Burkholder-Davis-Gundy inequality and \eqref{Y},
\begin{equation}
\begin{aligned}
\mathbb{E}\left[\left(\int_{t}^{t+\varepsilon}|\Delta\mathbf{Z}_{k}(s,i)|^{2}dt\right)^{\frac{p}{2}}\right]\leq&C_{p}\mathbb{E}\left[\sup_{t\leq s\leq t+\varepsilon}\left|\int_{t}^{s}\Delta\mathbf{Z}_{k}(s,i)dW(s)\right|^{p}\right]\\
\leq&C_{p}\mathbb{E}\left[\sup_{t\leq s\leq t+\varepsilon}|\Delta\mathbf{Y}_{k}(s,i)|^{p}\right.\\
&\qquad+\left.\left(\int_{t}^{t+\varepsilon}|f_{k}(s,i,\Delta\mathbf{X}_{k}(s,i),\Delta\mathbf{Y}_{k}(s,i),\Delta\mathbf{Z}_{k}(s,i))|ds\right)^{p}\right]\\
&+C_{p}K^{p}\textstyle\sum_{j\neq i}\mathbb{E}\left[\int_{t}^{t+\varepsilon}|\Delta\widetilde{P}_{k-1}(s,j)\mathbf{X}_{k-1}(s,i)|^{p}ds\right]\\
\leq&C_{p}\mathbb{E}\left[|\Delta\widetilde{P}_{k}(t+\varepsilon,i)\mathbf{X}_{k}(t+\varepsilon,i)|^{2}\right]\\
&+(C_{p}+C_{p}\varepsilon^{p})\mathbb{E}\left[\sup_{t\leq s\leq t+\varepsilon}|\Delta\mathbf{X}_{k}(s,i)|^{p}\right]\\
&+C_{p}\varepsilon^{\frac{p}{2}}\mathbb{E}\left[\left(\int_{t}^{t+\varepsilon}|\Delta\mathbf{Z}_{k}(s,i)|^{2}ds\right)^{\frac{p}{2}}\right]\\
&+C_{p}K^{p}\textstyle\sum_{j\neq i}\mathbb{E}\left[\int_{t}^{t+\varepsilon}|\Delta\widetilde{P}_{k-1}(s,j)\mathbf{X}_{k-1}(s,i)|^{p}ds\right],
\end{aligned}
\end{equation}
where $C_{p}$ stands for a positive constant depending on $p,L$ and the uniform bound of $\widetilde{P}_{k-1}$, changing from line to line. By choosing $0<\varepsilon\leq 1$ such that
\begin{equation}\label{varepsilon}
1-C_{p}\varepsilon^{\frac{p}{2}}>0,
\end{equation}
we can obtain that for any $0<\tau\leq\varepsilon$,
\begin{equation}\label{Z}
\begin{aligned}
\mathbb{E}\left[\left(\int_{0}^{\tau}|\Delta\mathbf{Z}_{k}(t,i)|^{2}dt\right)^{\frac{p}{2}}\right]\leq&C_{p}\mathbb{E}\left[|\Delta\widetilde{P}_{k}(\tau,i)\mathbf{X}_{k}(\tau,i)|^{p}\right]\\
&+(C_{p}+C_{p}\varepsilon^{p})\mathbb{E}\left[\sup_{0\leq t\leq\tau}|\Delta\mathbf{X}_{k}(t,i)|^{p}\right]\\
&+C_{p}K^{p}\textstyle\sum_{j\neq i}\mathbb{E}\left[\int_{0}^{\tau}|\Delta\widetilde{P}_{k-1}(t,j)\mathbf{X}_{k-1}(t,i)|^{p}dt\right].
\end{aligned}
\end{equation}
As for the forward part of \eqref{FBSDE-3}, we also use the Burkholder-Davis-Gundy inequality and obtain that
\begin{equation*}
\begin{aligned}
\mathbb{E}\left[\sup\limits_{0\leq t\leq\epsilon}|\Delta\mathbf{X}_{k}(t,i)|^{p}\right]\leq&C_{p}\mathbb{E}\left[\left(\int_{0}^{\epsilon}|b(t,i,\Delta\mathbf{X}_{k}(t,i),\Delta\mathbf{Y}_{k}(t,i),\Delta\mathbf{Z}_{k}(t,i))|dt\right)^{p}\right]\\
&+C_{p}\mathbb{E}\left[\left(\int_{0}^{\epsilon}|\sigma(t,i,\Delta\mathbf{X}_{k}(t,i),\Delta\mathbf{Y}_{k}(t,i),\Delta\mathbf{Z}_{k}(t,i))|^{2}dt\right)^{\frac{p}{2}}\right]\\
\leq&C_{p}(\epsilon^{p}+\epsilon^{\frac{p}{2}})\mathbb{E}\left[\sup\limits_{0\leq t\leq\epsilon}|\Delta\mathbf{X}_{k}(t,i)|^{p}+\sup\limits_{0\leq t\leq\epsilon}|\Delta\mathbf{Y}_{k}(t,i)|^{p}\right]\\
&+C_{p}(\epsilon^{\frac{p}{2}}+L_{\sigma,z}^{p})\mathbb{E}\left[\left(\int_{0}^{\epsilon}|\Delta\mathbf{Z}_{k}(t,i)|^{2}dt\right)^{\frac{p}{2}}\right]\\
\leq&C_{p}(\epsilon^{\frac{p}{2}}+L_{\sigma,z}^{p})\mathbb{E}\left[\sup\limits_{0\leq t\leq\epsilon}|\Delta\mathbf{X}_{k}(t,i)|^{p}\right]\\
&+C_{p}K^{p}(\epsilon^{\frac{p}{2}}+L_{\sigma,z}^{p})\textstyle\sum_{j\neq i}\mathbb{E}\left[\int_{0}^{\varepsilon}|\Delta\widetilde{P}_{k-1}(t,j)\mathbf{X}_{k-1}(t,i)|^{p}dt\right].
\end{aligned}
\end{equation*}
Let $\widehat{L}_{\sigma,p}>0$ be sufficiently small such that
\begin{equation}\label{condition1}
\begin{aligned}
C_{p}\widehat{L}_{\sigma,p}^{p}<1.
\end{aligned}
\end{equation}
When $L_{\sigma,z}\leq\widehat{L}_{\sigma,p}$, then there exists sufficiently small $\epsilon>0$ such that
\begin{equation}\label{epsilon}
\begin{aligned}
1-C_{p}(\epsilon^{\frac{p}{2}}+L_{\sigma,z}^{p})>0.
\end{aligned}
\end{equation}
Now we choose
$$
\tau=\varepsilon\wedge\epsilon
$$
and obtain that
\begin{equation*}
\begin{aligned}
\mathbb{E}\left[\sup\limits_{0\leq t\leq\tau}|\Delta\mathbf{X}_{k}(t,i)|^{p}\right]\leq&C_{p}\mathbb{E}\left[|\Delta\widetilde{P}_{k}(\tau,i)\mathbf{X}_{k}(\tau,i)|^{p}\right]\\
&+C_{p}K^{p}\textstyle\sum_{j\neq i}\mathbb{E}\left[\int_{0}^{\tau}|\Delta\widetilde{P}_{k-1}(t,j)\mathbf{X}_{k-1}(t,i)|^{p}dt\right]
\end{aligned}
\end{equation*}
where $C_{p}$ is still a positive constant only depending on $p,L$ and the uniform bound of $\widetilde{P}_{k-1}$. Thus there exists a constant $\widetilde{L}_{\sigma,p}$ such that 
\begin{equation*}
  \widetilde{L}_{\sigma,p}\leq \min\left\{\sigma_0,\bar{L}_{\sigma,2p},\widehat{L}_{\sigma,p}\right\}
\end{equation*}
and when $L_{\sigma,z}\leq \widetilde{L}_{\sigma,p}$, it holds that for any $i,j\in\mathcal{M}$,
\begin{align*}
&\mathbb{E}\left[\int_{0}^{\tau}|\Delta \widetilde{P}_{k-1}(t,j)\mathbf{X}_{k-1}(t,i)|^{p}dt\right]\leq \mathbb{E}\left[\int_{0}^{\tau}|\Delta \widetilde{P}_{k-1}(t,j)|^{2p}dt\right]^{\frac{1}{2}}\mathbb{E}\left[\sup\limits_{0\leq t\leq\tau}|\mathbf{X}_{k-1}(t,i)|^{2p}\right]^{\frac{1}{2}}\rightarrow0,\\
&\mathbb{E}\left[|\Delta\widetilde{P}_{k}(\tau,i)\mathbf{X}_{k-1}(\tau,i)|^{p}\right]\leq \mathbb{E}\left[|\Delta\widetilde{P}_{k}(\tau,i)|^{2p}\right]^{\frac{1}{2}}\mathbb{E}\left[\sup\limits_{0\leq t\leq \tau}|\mathbf{X}_{k-1}(t,i)|^{2p}\right]^{\frac{1}{2}}\rightarrow0~as~k\rightarrow\infty.
\end{align*}
Thus, we obtain
\begin{align*}
\lim_{k\rightarrow\infty}\mathbb{E}\left[\sup\limits_{0\leq t\leq\tau}|\Delta\mathbf{X}_{k}(t,i)|^{p}\right]=0.
\end{align*}
Similarly, from \eqref{Y} and \eqref{Z}, we can respectively obtain that for any $i\in\mathcal{M}$
$$
\lim_{k\rightarrow\infty}\mathbb{E}\left[\sup\limits_{0\leq t\leq\tau}|\Delta\mathbf{Y}_{k}(t,i)|^{p}\right]=0,~~\lim_{k\rightarrow\infty}\mathbb{E}\left[\left(\int_{0}^{\tau}|\Delta\mathbf{Z}_{k}(t,i)|^{2}dt\right)^{\frac{p}{2}}\right]=0.
$$
Similarly, it holds that
$$
\lim_{k\rightarrow\infty}\mathbb{E}\left[\sup\limits_{\tau\leq t\leq2\tau}|\Delta\mathbf{Y}_{k}(t,i)|^{p}\right]=0,\lim_{k\rightarrow\infty}\mathbb{E}\left[\sup\limits_{\tau\leq t\leq2\tau}|\Delta\mathbf{Y}_{k}(t,i)|^{p}\right]=0,\lim_{k\rightarrow\infty}\mathbb{E}\left[\left(\int_{\tau}^{2\tau}|\Delta\mathbf{Z}_{k}(t,i)|^{2}dt\right)^{\frac{p}{2}}\right]=0.
$$
By induction, we get
\begin{equation}\label{p}
\lim_{k\rightarrow\infty}\mathbb{E}\left[\sup\limits_{0\leq t\leq T}|\Delta\mathbf{Y}_{k}(t,i)|^{p}\right]=0,\lim_{k\rightarrow\infty}\mathbb{E}\left[\sup\limits_{0\leq t\leq T}|\Delta\mathbf{Y}_{k}(t,i)|^{p}\right]=0,\lim_{k\rightarrow\infty}\mathbb{E}\left[\left(\int_{0}^{T}|\Delta\mathbf{Z}_{k}(t,i)|^{2}dt\right)^{\frac{p}{2}}\right]=0.
\end{equation}Since
$$
\Delta\mathbf{u}_{k}(t,i)=-\widetilde{R}(t,i)^{-1}\left(B(t,i)^{\top}\Delta\mathbf{Y}_{k}(t,i)+D(t,i)^{\top}\Delta\mathbf{Z}_{k}(t,i)+\widetilde{S}(t,i)\Delta\mathbf{X}_{k}(t,i)\right),~ k\geq1,~i\in\mathcal{M},
$$
it is easy to verify that 
$$
\lim_{k\rightarrow\infty}\mathbb{E}\left[\left(\int_{0}^{T}|\Delta\mathbf{u}_{k}(t,i)|^{2}dt\right)^{\frac{p}{2}}\right]=0.
$$
$\textbf{Step 4. (Convergence of $\left(P_{k},\Lambda_{k}\right)$ in $L_{\mathcal{F}^{W}}^{r}(\Omega;C(0,T;\mathbb{S}^{n}))\times L_{\mathcal{F}^{W}}^{r}(0,T;\mathbb{S}^{n})$ for some $r>2$)}$ For $k\geq1,i\in\mathcal{M}$, it follows from \cite{Gal 1979} and \cite{Tang 2003} that the inverse of $\mathbf{X}_{k}(t,i)$ (denoted by $\mathbf{X}_{k}(t,i)^{-1}$) exists and satisfies the following linear matrix-valued SDE:
\begin{equation}\label{X-1}
\left\{
\begin{array}{l}
d\mathbf{X}_{k}(t,i)^{-1}=-\mathbf{X}_{k}(t,i)^{-1}\left[A_{k}(t,i)-C_{k}(t,i)^{2}\right]dt-\mathbf{X}_{k}(t,i)^{-1}C_{k}(t,i)dW(t),\quad t\in[0,T],\\
\mathbf{X}_{k}(0,i)^{-1}=\mathbf{I}_{n},
\end{array}
\right.
\end{equation}
where
\begin{align*}
A_{k}(t,i)=&A(t,i)-B(t,i)[\widetilde{R}(t,i)+D(t,i)\widetilde{P}_{k}(t,i)D(t,i)^{\top}]^{-1}\\
&\cdot[B(t,i)^{\top}\widetilde{P}_{k}(t,i)+D(t,i)^{\top}\widetilde{P}_{k}(t,i)C(t,i)+D(t,i)^{\top}\widetilde{\Lambda}_{k}(t,i)+\widetilde{S}(t,i)],\\
C_{k}(t,i)=&C(t,i)-D(t,i)[\widetilde{R}(t,i)+D(t,i)\widetilde{P}_{k}(t,i)D(t,i)^{\top}]^{-1}\\
&\cdot[B(t,i)^{\top}\widetilde{P}_{k}(t,i)+D(t,i)^{\top}\widetilde{P}_{k}(t,i)C(t,i)+D(t,i)^{\top}\widetilde{\Lambda}_{k}(t,i)+\widetilde{S}(t,i)].
\end{align*}
We consider the following SDEs:
\begin{equation}\label{norm-0}
\left\{
\begin{array}{l}
\begin{aligned}
d\left|\mathbf{X}_{k}(t,i)^{-1}\right|^{2}=&tr\left(\left[\mathbf{X}_{k}(t,i)^{-1}\right]^{\top}\mathbf{X}_{k}(t,i)^{-1}\left[C_{k}(t,i)C_{k}(t,i)^{\top}-A_{k}(t,i)-A_{k}(t,i)^{\top}-C_{k}(t,i)^{2}\right.\right.\\
&\left.\left.-(C_{k}(t,i)^{2})^{\top}\right]\right) dt-tr\left(\left[\mathbf{X}_{k}(t,i)^{-1}\right]^{\top}\mathbf{X}_{k}(t,i)^{-1}\left[C_{k}(t,i)+C_{k}(t,i)^{\top}\right]\right) dW(t),
\end{aligned}\\
\left|\mathbf{X}_{k}(0,i)^{-1}\right|^{2}=n,\quad k\geq1,\quad i\in\mathcal{M},\quad t\in[0,T].
\end{array}
\right.
\end{equation}
From assumptions and the uniform boundedness of $\widetilde{P}_{k}(i)$, we rewrite \eqref{norm-0} as
\begin{equation}\label{norm-1}
\left\{
\begin{array}{l}
\begin{aligned}
d\left|\mathbf{X}_{k}(t,i)^{-1}\right|^{2}=&\left|\mathbf{X}_{k}(t,i)^{-1}\right|^{2}\left(\alpha_{k}(t,i)L_{\sigma,z}^{2}|\widetilde{\Lambda}_{k}(t,i)|^{2}+\beta_{k}(t,i)|\widetilde{\Lambda}_{k}(t,i)|+\gamma_{k}(t,i)\right)dt\\
&+\left|\mathbf{X}_{k}(t,i)^{-1}\right|^{2}\left(\delta_{k}(t,i)L_{\sigma,z}|\widetilde{\Lambda}_{k}(t,i)|+\eta_{k}(t,i)\right)dW(t),
\end{aligned}\\
\left|\mathbf{X}_{k}(0,i)^{-1}\right|^{2}=n,\quad k\geq1,\quad i\in\mathcal{M},\quad t\in[0,T]
\end{array}
\right.
\end{equation}
by Lemma \ref{le-a}, where $\alpha_{k}(t,i),\beta_{k}(t,i),\gamma_{k}(t,i),\delta_{k}(t,i),\eta_{k}(t,i)$ are uniformly bounded with respect to $i$ and $k$. Then we can obtain that
\begin{align*}
&\left|\mathbf{X}_{k}(t,i)^{-1}\right|^{2}\\
=&n\exp\left\{\int_{0}^{t}\left[\alpha_{k}(s,i)L_{\sigma,z}^{2}|\widetilde{\Lambda}_{k}(s,i)|^{2}+\beta_{k}(s,i)|\widetilde{\Lambda}_{k}(s,i)|+\gamma_{k}(s,i)-\frac{1}{2}\left(\delta_{k}(s,i)L_{\sigma,z}|\widetilde{\Lambda}_{k}(s,i)|+\eta_{k}(s,i)\right)^{2}\right]ds\right.\\
&\left.\qquad+\int_{0}^{t}\left(\delta_{k}(s,i)L_{\sigma,z}|\widetilde{\Lambda}_{k}(s,i)|+\eta_{k}(s,i)\right)dW(s)\right\}.
\end{align*}
Take $q_2>2$, and let $\tilde{L}_{\sigma}>0$ be sufficiently small such that for each $i\in\mathcal{M}$ and $k\geq1$,
\begin{equation*}
\begin{aligned}
&\|\delta_{k}(\cdot,i)\|_{L_{\mathcal{F}^{W}}^{\infty}(0,T;\mathbb{R})}\tilde{L}_{\sigma}\|\widetilde{\Lambda}_{k}(i)\|_{L_{\mathcal{F}^{W}}^{2,\mathrm{BMO}}(0,T;\mathbb{S}^{n})}<\Phi(q_{2}),\\
&\sqrt{q_{2}\|\alpha_{k}(\cdot,i)\|_{L_{\mathcal{F}^{W}}^{\infty}(0,T;\mathbb{R})}}\tilde{L}_{\sigma}\|\widetilde{\Lambda}_{k}(i)\|_{L_{\mathcal{F}^{W}}^{2,\mathrm{BMO}}(0,T;\mathbb{S}^{n})}<1,
\end{aligned}
\end{equation*}
where
$$
\Phi(x)=\sqrt{1+\frac{1}{x^{2}}\log(\frac{2x-1}{2(x-1)})}-1,~1<x<\infty.
$$
When $L_{\sigma,z}\leq \tilde{L}_{\sigma}$, by H\"older's inequality it holds that for $2<q<q_{1}<q_{2}$,
\begin{align}\label{q}
&\mathbb{E}\left[\sup_{0\leq t\leq T}|\mathbf{X}_{k}(t,i)^{-1}|^{q}\right]\nonumber\\
=&\mathbb{E}\left[\sup_{0\leq t\leq T}n^{\frac{q}{2}}\exp\left\{\int_{0}^{t}\left[\frac{q}{2}\gamma_{k}(s,i)-\frac{q}{4}\eta_{k}(s,i)^{2}\right]ds+\int_{0}^{t}\frac{q}{2}\eta_{k}(s,i)dW(s)\right\}\cdot\exp\left\{\int_{0}^{t}\frac{q}{2}\delta_{k}(s,i)L_{\sigma}|\widetilde{\Lambda}_{k}(s,i)|dW(s)\right\}\right.\nonumber\\
&\cdot\left.\exp\left\{\int_{0}^{t}\left[\left(\frac{q}{2}\alpha_{k}(s,i)-\frac{q}{4}\delta_{k}(s,i)^{2}\right)L_{\sigma}^{2}|\widetilde{\Lambda}_{k}(s,i)|^{2}+\frac{q}{2}\left(\beta_{k}(s,i)-\delta_{k}(s,i)\eta_{k}(s,i)L_{\sigma}\right)|\widetilde{\Lambda}_{k}(s,i)|\right]ds\right\}\right]\nonumber\\
\leq&N\mathbb{E}\left[\sup_{0\leq t\leq T}\exp\left\{\int_{0}^{t}\frac{q_{1}}{2}\delta_{k}(s,i)L_{\sigma}|\widetilde{\Lambda}_{k}(s,i)|dW(s)\right\}\right.\nonumber\\
&\cdot\left.\exp\left\{\int_{0}^{t}\left[\left(\frac{q_{1}}{2}\alpha_{k}(s,i)-\frac{q_{1}}{4}\delta_{k}(s,i)^{2}\right)L_{\sigma}^{2}|\widetilde{\Lambda}_{k}(s,i)|^{2}+\frac{q_{1}}{2}\left(\beta_{k}(s,i)-\delta_{k}(s,i)\eta_{k}(s,i)L_{\sigma}\right)|\widetilde{\Lambda}_{k}(s,i)|\right]ds\right\}\right]^{\frac{q}{q_1}}\nonumber\\
\leq&N\mathbb{E}\left[\sup_{0\leq t\leq T}\left\{\mathcal{E}_{t}\left(\delta(i)L_{\sigma}|\widetilde{\Lambda}_{k}(i)|\cdot W\right)\right\}^{\frac{q_{1}}{2}}\cdot\exp\left\{\int_{0}^{t}\frac{q_{1}}{2}\alpha_{k}(s,i)L_{\sigma}^{2}|\widetilde{\Lambda}_{k}(s,i)|^{2}ds\right\}\right.\nonumber\\
&\cdot\left.\exp\left\{\int_{0}^{t}\left[\frac{q_1 \kappa^{2}}{2}|\widetilde{\Lambda}_{k}(s,i)|^{2}+\frac{q_{1}\left(\beta_{k}(s,i)-\delta_{k}(s,i)\eta_{k}(s,i)L_{\sigma}\right)^{2}}{8\kappa^{2}}\right]ds\right\}\right]^{\frac{q}{q_1}}\nonumber\\
\leq&N\mathbb{E}\left[\sup_{0\leq t\leq T}\left\{\mathcal{E}_{t}\left(\delta(i)L_{\sigma}|\widetilde{\Lambda}_{k}(i)|\cdot W\right)\right\}^{\frac{q_{2}}{2}}\cdot\exp\left\{\int_{0}^{t}\frac{q_{2}}{2}\left(\alpha(s)L_{\sigma}^{2}+\kappa^2\right)|\widetilde{\Lambda}_{k}(s,i)|^{2}ds\right\}\right]^{\frac{q}{q_2}}\nonumber\\
\leq&N\mathbb{E}\left[\sup_{0\leq t\leq T}\left\{\mathcal{E}_{t}\left(\delta(i)L_{\sigma}|\widetilde{\Lambda}_{k}(i)|\cdot W\right)\right\}^{q_{2}}\right]^{\frac{q}{2q_2}}\mathbb{E}\left[\sup_{0\leq t\leq T}\exp\left\{\int_{0}^{t}q_{2}\alpha_{k}(s,i)L_{\sigma}^{2}|\widetilde{\Lambda}_{k}(s,i)|^{2}ds\right\}\right]^{\frac{q}{2q_2}}\nonumber\\
\leq&N,
\end{align}
where $\kappa$ is a sufficiently small positive constant, $N$ is positive constant changing from line to line and the last inequality follows from \cite[Theorem 2.2 and Theorem 3.1]{Kazamaki 1994}.

Recalling \eqref{Y=PX}, for $k\geq 1,i\in\mathcal{M},t\in[0,T]$, we have
\begin{equation*}
\begin{aligned}
\widetilde{P}_{k}(t,i)=&\mathbf{Y}_{k}(t,i)\mathbf{X}_{k}(t,i)^{-1},\\
\widetilde{\Lambda}_{k}(t,i)=&\mathbf{Z}_{k}(t,i)\mathbf{X}_{k}(t,i)^{-1}-\widetilde{P}_{k}(t,i)C(t,i)-\widetilde{P}_{k}(t,i)D(t,i)\mathbf{u}_{k}(t,i)\mathbf{X}_{k}(t,i)^{-1}
\end{aligned}
\end{equation*}
and furthermore
\begin{equation*}
\begin{aligned}
\Delta\widetilde{P}_{k+1}(t,i)=&\Delta\mathbf{Y}_{k+1}(t,i)\mathbf{X}_{k+1}(t,i)^{-1}-\widetilde{P}_{k}(t,i)\Delta\mathbf{X}_{k+1}(t,i)\mathbf{X}_{k+1}(t,i)^{-1},\\
\Delta \widetilde{\Lambda}_{k+1}(t,i)=&\Delta\mathbf{Z}_{k+1}(t,i)\mathbf{X}_{k+1}(t,i)^{-1}-\mathbf{Z}_{k}(t,i)\mathbf{X}_{k}(t,i)^{-1}\Delta\mathbf{X}_{k+1}(t,i)\mathbf{X}_{k+1}(t,i)^{-1}\\
&-\Delta\widetilde{P}_{k+1}(t,i)C(t,i)-\Delta \widetilde{P}_{k+1}(t,i)D(t,i)\mathbf{u}_{k+1}(t,i)\mathbf{X}_{k+1}(t,i)^{-1}\\
&-\widetilde{P}_{k}(t,i)D(t,i)\Delta\mathbf{u}_{k+1}(t,i)\mathbf{X}_{k+1}(t,i)^{-1}\\
&+\widetilde{P}_{k}(t,i)D(t,i)\mathbf{u}_{k}(t,i)\mathbf{X}_{k}(t,i)^{-1}\Delta\mathbf{X}_{k+1}(t,i)\mathbf{X}_{k+1}(t,i)^{-1}\\
=&\Delta\mathbf{Z}_{k+1}(t,i)\mathbf{X}_{k+1}(t,i)^{-1}-\widetilde{\Lambda}_{k}(t,i)\Delta\mathbf{X}_{k+1}(t,i)\mathbf{X}_{k+1}(t,i)^{-1}\\
&-\Delta\widetilde{P}_{k+1}(t,i)C(t,i)-\Delta\widetilde{P}_{k+1}(t,i)D(t,i)\mathbf{u}_{k+1}(t,i)\mathbf{X}_{k+1}(t,i)^{-1}\\
&-\widetilde{P}_{k}(t,i)D(t,i)\Delta\mathbf{u}_{k+1}(t,i)\mathbf{X}_{k+1}(t,i)^{-1}-\widetilde{P}_{k}(t,i)C(t,i)\Delta\mathbf{X}_{k+1}(t,i)\mathbf{X}_{k+1}(t,i)^{-1}\\
\end{aligned}
\end{equation*}
 Taking some $p>2$ and letting $r'=pq/(p+q)$, $2<r<r'$ and $u>2$ satisfying $\frac{1}{p}+\frac{1}{q}+\frac{1}{u}=\frac{1}{r}$. Then there exists a constant $L_\sigma>0$ such that
 \begin{equation*}
  L_{\sigma}\leq \min\left\{\widetilde{L}_{\sigma,p},\tilde{L}_{\sigma}\right\}
 \end{equation*}
 and when $L_{\sigma,z}\leq L_\sigma$, by H\"older's inequality, it holds that
\begin{align*}
\mathbb{E}\left[\sup_{0\leq t\leq T}\left|\Delta\widetilde{P}_{k+1}(t,i)\right|^{r}\right]\leq&\mathbb{E}\left[\sup_{0\leq t\leq T}\left|\Delta\widetilde{P}_{k+1}(t,i)\right|^{r'}\right]^{\frac{r}{r'}}\\
\leq&a_{r'}\mathbb{E}\left[\sup_{0\leq t\leq T}\left|\Delta\mathbf{Y}_{k+1}(t,i)\mathbf{X}_{k+1}(t, i)^{-1}\right|^{r'}\right]^{\frac{r}{r'}}\\
&\quad+a_{r'}\mathbb{E}\left[\sup_{0\leq t\leq T}\left|\Delta\mathbf{X}_{k+1}(t,i)\mathbf{X}_{k+1}(t,i)^{-1}\right|^{r'}\right]^{\frac{r}{r'}}\\
\leq&a_{r'}\mathbb{E}\left[\sup_{0\leq t\leq T}\left|\Delta\mathbf{Y}_{k+1}(t,i)\right|^{p}\right]^{\frac{r}{p}}\mathbb{E}\left[\sup_{0\leq t\leq T}\left|\mathbf{X}_{k+1}(t, i)^{-1}\right|^{q}\right]^{\frac{r}{q}}\\
&\quad+a_{r'}\mathbb{E}\left[\sup_{0\leq t\leq T}\left|\Delta\mathbf{X}_{k+1}(t,i)\right|^{p}\right]^{\frac{r}{p}}\mathbb{E}\left[\sup_{0\leq t\leq T}\left|\mathbf{X}_{k+1}(t, i)^{-1}\right|^{q}\right]^{\frac{r}{q}}
\end{align*}
and
\begin{align*}
&\mathbb{E}\left[\left(\int_{0}^{T}\left|\Delta\widetilde{\Lambda}_{k+1}(t,i)\right|^{2}dt\right)^{\frac{r}{2}}\right]\\
\leq&b_{r}\left\{\mathbb{E}\left[\left(\int_{0}^{T}\left|\Delta\mathbf{Z}_{k+1}(t,i)\mathbf{X}_{k+1}(t,i)^{-1}\right|^{2}dt\right)^{\frac{r}{2}}\right]
+\mathbb{E}\left[\left(\int_{0}^{T}\left|\widetilde{\Lambda}_{k}(t,i)\Delta\mathbf{X}_{k+1}(t,i)\mathbf{X}_{k+1}(t,i)^{-1}\right|^{2}dt\right)^{\frac{r}{2}}\right]\right.\\
&+\mathbb{E}\left[\left(\int_{0}^{T}\left|\Delta\widetilde{P}_{k+1}(t,i)C(t,i)\right|^{2}dt\right)^{\frac{r}{2}}\right]
+\mathbb{E}\left[\left(\int_{0}^{T}\left|\Delta\widetilde{P}_{k+1}(t,i)D(t,i)\mathbf{u}_{k+1}(t,i)\mathbf{X}_{k+1}(t,i)^{-1}\right|^{2}dt\right)^{\frac{r}{2}}\right]\\
&+\mathbb{E}\left[\left(\int_{0}^{T}\left|\widetilde{P}_{k}(t,i)D(t,i)\Delta\mathbf{u}_{k+1}(t,i)\mathbf{X}_{k+1}(t,i)^{-1}\right|^{2}dt\right)^{\frac{r}{2}}\right]\\
&+\left.\mathbb{E}\left[\left(\int_{0}^{T}\left|\widetilde{P}_{k}(t,i)C(t,i)\Delta\mathbf{X}_{k+1}(t,i)\mathbf{X}_{k+1}(t,i)^{-1}\right|^{2}dt\right)^{\frac{r}{2}}\right]\right\}\\
\leq&b_{r}\left\{\mathbb{E}\left[\left(\int_{0}^{T}\left|\Delta\mathbf{Z}_{k+1}(t,i)\right|^{2}dt\right)^\frac{p}{2}\right]^{\frac{r}{p}}\mathbb{E}\left[\sup_{0\leq t\leq T}\left|\mathbf{X}_{k+1}(t,i)^{-1}\right|^{q}\right]^{\frac{r}{q}}\right.\\
&+\mathbb{E}\left[\left(\int_{0}^{T}\left|\widetilde{\Lambda}_{k}(t,i)\right|^{2}ds\right)^\frac{u}{2}\right]^{\frac{r}{u}}\mathbb{E}\left[\sup_{0\leq t\leq T}\left|\Delta\mathbf{X}_{k+1}(t, i)\right|^{p}\right]^{\frac{r}{p}}\mathbb{E}\left[\sup_{0\leq t\leq T}\left|\mathbf{X}_{k+1}(t,i)^{-1}\right|^{q}\right]^{\frac{r}{q}}\\
&+\mathbb{E}\left[\sup_{0\leq t\leq T}\left|\Delta\widetilde{P}_{k+1}(t,i)\right|^{r}\right]\\
&+\mathbb{E}\left[\left(\int_{0}^{T}\left|\Delta\widetilde{P}_{k+1}(t,i)\right|^{u}ds\right)\right]^{\frac{r}{u}}\mathbb{E}\left[\sup_{0\leq t\leq T}\left|\mathbf{u}_{k+1}(t, i)\right|^{p}\right]^{\frac{r}{p}}\mathbb{E}\left[\sup_{0\leq t\leq T}\left|\mathbf{X}_{k+1}(t,i)^{-1}\right|^{q}\right]^{\frac{r}{q}}\\
&+\mathbb{E}\left[\left(\int_{0}^{T}\left|\Delta\mathbf{u}_{k+1}(t,i)\right|^{2}dt\right)^\frac{p}{2}\right]^{\frac{r}{p}}\mathbb{E}\left[\sup_{0\leq t\leq T}\left|\mathbf{X}_{k+1}(t,i)^{-1}\right|^{q}\right]^{\frac{r}{q}}\\
&+\left.\mathbb{E}\left[\sup_{0\leq t\leq T}\left|\Delta\mathbf{X}_{k+1}(t,i)\right|^{p}\right]^{\frac{r}{p}}\mathbb{E}\left[\sup_{0\leq t\leq T}\left|\mathbf{X}_{k+1}(t,i)^{-1}\right|^{q}\right]^{\frac{r}{q}}\right\}
\end{align*}
where $a_{r'},b_{r}$ are deterministic positive constants, changing from line to line.

Recalling the convergence of $\left\{\mathbf{X}_{k}(i),\mathbf{Y}_{k}(i),\mathbf{Z}_{k}(i),\mathbf{u}_{k}(i)\right\}_{k\geq1}$ establied in Step 3, we get that $\left(P_{k}(i),\Lambda_{k}(i)\right)$ converges in $L_{\mathcal{F}^{W}}^{r}(\Omega;C(0,T;\mathbb{S}^{n}))\times L_{\mathcal{F}^{W}}^{r}(0,T;\mathbb{S}^{n})$ as $k\rightarrow\infty$ for each $i\in\mathcal{M}$ and $r>2$, i.e. there exists a pair of adapted processes $\left(P(i),\Lambda(i)\right)$ such that
$$
\lim_{k\rightarrow\infty}\mathbb{E}\left[\sup_{0\leq t\leq T}\left|\widetilde{P}_{k}(t,i)-\widetilde{P}(t,i)\right|^{r}\right]=0,\quad \lim_{k\rightarrow\infty}\mathbb{E}\left[\left(\int_{0}^{T}\left|\widetilde{\Lambda}_{k}(t,i)-\widetilde{\Lambda}(t,i)\right|^{2}dt\right)^{\frac{r}{2}}\right]=0.
$$
Now letting $k\rightarrow\infty$ in \eqref{overlinePk} and taking the following transform
$$
P(t,i)=e^{-q_{ii}t}\widetilde{P}(t,i),\quad\Lambda(t,i)=e^{-q_{ii}t}\widetilde{\Lambda}(t,i),\qquad t\in[0,T],~i\in\mathcal{M},
$$
we obtain the existence of the solution for ESRE \eqref{Riccati}.\\
\\
$\textbf{Step 5. ($\left(P_{k},\Lambda_{k}\right)$ belongs to $L_{\mathcal{F}^{W}}^{\infty}(0,T;\mathbb{S}^{n})\times L_{\mathcal{F}^{W}}^{2,\mathrm{BMO}}(0,T;\mathbb{S}^{n})$)}$ We consider the following equivalent form of ESRE \eqref{Riccati}:
\begin{equation}\label{Riccati-1}
\left\{
\begin{array}{l}
\begin{aligned}
d\widetilde{P}(t,i)=&-[\Pi(t,i,\widetilde{P},\widetilde{\Lambda})+\widetilde{Q}(t,i)+H(t,i,\widetilde{P},\widetilde{\Lambda},\widetilde{R},\widetilde{S})\\
&\quad+\textstyle\sum_{j\neq i}q_{ij}e^{(q_{ii}-q_{jj})t}\widetilde{P}(t,j)]dt+\widetilde{\Lambda}(t,i)dW(t),
\end{aligned}\\
\widetilde{P}(T,i)=\widetilde{G}(i),\quad i\in\mathcal{M},~t\in[0,T],
\end{array}
\right.
\end{equation}
Similar to \eqref{rho}, we can obtain that for any $\beta>0$,
\begin{equation*}
\begin{aligned}
e^{\beta t}|\widetilde{P}(t,i)|^{2}=&e^{\beta T}|\widetilde{G}(i)|^{2}+\int_{t}^{T}e^{\beta s}\left[-\beta|\widetilde{P}(s,i)|^{2}-|\widetilde{\Lambda}(s,i)|^{2}+2\langle \widetilde{P}(s,i),\Pi(s,i,\widetilde{P},\widetilde{\Lambda})+\widetilde{Q}(s,i)\right.\\
&\left.+H(s,i,\widetilde{P},\widetilde{\Lambda},\widetilde{R},\widetilde{S})+\textstyle\sum_{j\neq i}q_{ij}e^{(q_{ii}-q_{jj})s}\widetilde{P}(s,j)\rangle\right]ds-2\int_{t}^{T}e^{\beta s}\langle\widetilde{P}(s,i),\widetilde{\Lambda}(s,i)\rangle dW(s).
\end{aligned}
\end{equation*}
Taking conditional expectation, it holds for any $t\in[0,T]$ and each $i\in\mathcal{M}$ that
\begin{align*}
&e^{\beta t}|\widetilde{P}(t,i)|^{2}+\beta \mathbb{E}_{t}\left[\int_{t}^{T}e^{\beta s}|\widetilde{P}(s,i)|^{2}ds\right]+\mathbb{E}_{t}\left[\int_{t}^{T}e^{\beta s}|\widetilde{\Lambda}(s,i)|^{2}ds\right]\\
&\leq \mathbb{E}_{t}\left[e^{\beta T}|\widetilde{G}(i)|^{2}\right]+2K\mathbb{E}_{t}\left[\int_{t}^{T}e^{\beta s}|\widetilde{P}(s,i)|\left(\textstyle\sum_{j=1}^{l}|\widetilde{P}(s,j)|+|\widetilde{\Lambda}(s,i)|+1\right)ds\right]\\
&\leq \mathbb{E}_{t}\left[e^{\beta T}|\widetilde{G}(i)|^{2}\right]+(3K^{2}+lK)\mathbb{E}_{t}\left[\int_{t}^{T}e^{\beta s}|\widetilde{P}(s,i)|^{2}ds\right]+\frac{1}{2}\mathbb{E}_{t}\left[\int_{t}^{T}e^{\beta s}|\widetilde{\Lambda}(s,i)|^{2}ds\right]\\
&\quad+K\textstyle\sum_{j=1}^{l}\mathbb{E}_{t}\left[\int_{t}^{T}e^{\beta s}|\widetilde{P}(s,j)|^{2}ds\right]+\mathbb{E}_{t}\left[\int_{t}^{T}e^{\beta s}ds\right].
\end{align*}
We choose $\beta=3K^{2}+2lK$ and immediately obtain that
\begin{align*}
\sum\limits_{i=1}^{l}e^{\beta t}|\widetilde{P}(t,i)|^{2}+\frac{1}{2}\sum\limits_{i=1}^{l}\mathbb{E}_{t}\left[\int_{t}^{T}e^{\beta s}|\widetilde{\Lambda}(s,i)|^{2}ds\right]\leq le^{\beta T}(K^{2}+\frac{1}{\beta}),\quad \forall t\in[0,T],
\end{align*}
which implies $(\widetilde{P}(i),\widetilde{\Lambda}(i))$ belongs to $L_{\mathcal{F}^{W}}^{\infty}(0,T;\mathbb{S}^{n})\times L_{\mathcal{F}^{W}}^{2,\mathrm{BMO}}(0,T;\mathbb{S}^{n})$, so does $(P(i),\Lambda(i))$ for any $i\in\mathcal{M}$.\\
$\textbf{Step 6. ($P(i)\geq 0$)}$ Finally, $P(t,i)=\lim\limits_{k\rightarrow\infty}e^{-q_{ii}t}\widetilde{P}_{k}(t,i)\geq0$ for $i\in\mathcal{M},t\in[0,T]$.
\end{proof}

\section{Optimal control of Problem(SLQ)}
From the construction of a solution of ESRE \eqref{Riccati}, we obatin the optimal control in feedback form and show that it is indeed admissible. Thus we finally solve Problem(SLQ).\\ 
\begin{proof}[Proof\ of\ Theorem \ref{th-b}:] We divide the proof into two steps.\\
$\textbf{Step 1. ($u^{*}$ is an admissible control)}$ We write \eqref{optimal} as
\begin{equation}\label{star}
\begin{aligned}
u^{*}(t,i,X)=&-\left(\widetilde{R}(t,i)+D(t,i)^{\top}\widetilde{P}(t,i)D(t,i)\right)^{-1}\\
&\cdot\left(B(t,i)^{\top}\widetilde{P}(t,i)+D(t,i)^{\top}\widetilde{P}(t,i)C(t,i)+D(t,i)^{\top}\widetilde{\Lambda}(t,i)+\widetilde{S}(t,i)\right)X
\end{aligned}
\end{equation}
where $(\widetilde{P}(i),\widetilde{\Lambda}(i))_{i=1}^{l}$ is the solution of the SRE \eqref{Riccati-1}. Substituting \eqref{star} into the state process \eqref{state}, we have
\begin{equation*}
\left\{\begin{array}{l}
\begin{aligned}
dX(t)=&\left[A\left(t,\alpha_{t}\right)-B\left(t,\alpha_{t}\right)\left(\widetilde{R}(t,\alpha_{t})+D(t,\alpha_{t})^{\top}\widetilde{P}(t,\alpha_{t})D(t,\alpha_{t})\right)^{-1}\left(B(t,i)^{\top}\widetilde{P}(t,i)\right.\right.\\
&+\left.\left.D(t,i)^{\top}\widetilde{P}(t,i)C(t,i)+D(t,i)^{\top}\widetilde{\Lambda}(t,i)+\widetilde{S}(t,i)\right)\right]X(t)dt\\
&+\left[C\left(t,\alpha_{t}\right)-D\left(t,\alpha_{t}\right)\left(\widetilde{R}(t,\alpha_{t})+D(t,\alpha_{t})^{\top}\widetilde{P}(t,\alpha_{t})D(t,\alpha_{t})\right)^{-1}\left(B(t,i)^{\top}\widetilde{P}(t,i)\right.\right.\\
&\left.\left.+D(t,\alpha_t)^{\top}\widetilde{P}(t,\alpha_t)C(t,\alpha_t)+D(t,\alpha_t)^{\top}\widetilde{\Lambda}(t,i)+\widetilde{S}(t,i)\right)\right]X(t)dW(t),~t\in[0,T],
\end{aligned}\\
X(0)=x, \alpha_{0}=i_{0},
\end{array}\right.
\end{equation*}
Comparing the coefficients of the above equation with ones in \eqref{A,C}, we deduce from the proof of Lemma \ref{le-c} that for some $q>2$,
$$
\mathbb{E}\left[\sup_{0\leq t\leq T}|X(t)|^{q}\right]<\infty.
$$
Therefore, we have
\begin{align*}
&\mathbb{E}\left[\int_{0}^{T}|u^{*}(t,\alpha_{t})|^{2}dt\right]\\
\leq&C\mathbb{E}\left[\int_{0}^{T}\left|(1+|\widetilde{P}(t,\alpha_{t})|+|\widetilde{\Lambda}(t,\alpha_{t})|)|X^{*}(t,\alpha_{t})|\right|^{2}dt\right]\\
\leq&C\mathbb{E}\left[\sup\limits_{0\leq t\leq T}|X^{*}(t,\alpha_{t})|^{2}\right]+C\mathbb{E}\left[\left(\int_{0}^{T}|\widetilde{\Lambda}(t,\alpha_{t})|^{2}dt\right)\sup\limits_{0\leq t\leq T}|X^{*}(t,\alpha_{t})|^{2}\right]\\
\leq&C\mathbb{E}\left[\sup\limits_{0\leq t\leq T}|X^{*}(t,\alpha_{t})|^{2}\right]+C\mathbb{E}\left[\left(\int_{0}^{T}|\widetilde{\Lambda}(t,\alpha_{t})|^{2}dt\right)^{\frac{q}{q-2}}\right]^{\frac{q-2}{q}}\mathbb{E}\left[\sup\limits_{0\leq t\leq T}|X^{*}(t,\alpha_{t})|^{q}\right]^{\frac{2}{q}}\\
<&\infty
\end{align*}
where $C$ is a positive constant which can change from line to line.
\\
$\textbf{Step 2. ($u^{*}$ is an optimal control)}$ Applying It\^o's formula to $\langle P(t,\alpha_{t})X(t),X(t)\rangle$, we have
\begin{align*}
&\int_{0}^{T}d\langle P(t,\alpha_{t})X(t),X(t)\rangle\\
=&\int_{0}^{T}\langle dP(t,\alpha_{t})X(t),X(t)\rangle+\int_{0}^{T}\langle P(t,\alpha_{t})X(t),dX(t)\rangle+\int_{0}^{T}\langle dP(t,\alpha_{t})X(t),dX(t)\rangle\\
&+\int_{0}^{T}\langle \sum_{j,j'\in\mathcal{M}}(P(t,j)-P(t,j'))I_{\{\alpha_{t-}=j'\}}X(t),X(t)\rangle dN_{t}^{j'j}\\
=&\int_{0}^{T}\left\{-\left\langle P(t,\alpha_{t})A(t,\alpha_{t})X(t)+A(t,\alpha_{t})^{\top}P(t,\alpha_{s})X(t)+C(t,\alpha_{t})^{\top}P(t,\alpha_{t})C(t,\alpha_{t})X(t)\right.\right.\\
&+\Lambda(t,\alpha_{t})C(t,\alpha_{t})X(t)+C(t,\alpha_{t})^{\top}\Lambda(t,\alpha_{t})X(t)+Q(t,\alpha_{t})X(t)+\textstyle\sum_{j=1}^{l}q_{\alpha_{t}j}P(t,j)X(t)\\
&+H(t,\alpha_{t},P,\Lambda,S,R)X(t),X(t)\rangle+\langle P(t,\alpha_{t})\left(A(t,\alpha_{t})X(t)+B(t,\alpha_{t})u(t)\right),X(t)\rangle\\
&\left.+\langle \Lambda(t,\alpha_{t})\left(C(t,\alpha_{t})X(t)+D(t,\alpha_{t})u(t)\right),X(t)\rangle\right\}dt+\int_{0}^{T}\langle P(t,\alpha_{t})X(t),A(t,\alpha_{t})X(t)+B(t,\alpha_{t})u(t)\rangle dt\\
&+\int_{0}^{T}\langle \Lambda(t,\alpha_{t})X(t)+P(t,\alpha_{t})\left(C(t,\alpha_{t})X(t)+D(t,\alpha_{t})u(t)\right),C(t,\alpha_{t})X(t)+D(t,\alpha_{t})u(t)\rangle dt\\
&+\int_{0}^{T}\langle \sum_{j,j'\in\mathcal{M}}(P(t,j)-P(t,j'))I_{\{\alpha_{t-}=j'\}}X(t),X(t)\rangle dN_{t}^{j'j}+\int_{0}^{T}\langle\Lambda(t,\alpha_t)X(t),X(t)\rangle dW(t)\\
&+\int_{0}^{T}2\langle P(t,\alpha_t)[C(t,\alpha_t)X(t)+D(t,\alpha_t)u(t)],X(t)\rangle dW(t)\\
=&\int_{0}^{T}\left\{-\langle Q(t,\alpha_{t})X(t),X(t)\rangle+\langle P(t,\alpha_{t})\left(C(t,\alpha_{t})X(t)+D(t,\alpha_{t})u(t)\right),C(t,\alpha_{t})X(t)+D(t,\alpha_{t})u(t)\rangle\right.\\
&-\langle H(t,\alpha_{t},P,\Lambda,S,R)X(t),X(t)\rangle+2\langle D(t,\alpha_{t})^{\top}\Lambda(t,\alpha_{t})X(t)+B(t,\alpha_{t})^{\top}P(t,\alpha_{t}),u(t)\rangle\\
&\left.-\langle P(t,\alpha_{t})C(t,\alpha_{t})X(t),C(t,\alpha_{t})X(t)\rangle\right\}ds\Bigg]+\int_{0}^{T}\langle\sum_{j=1}^{l}(P(t,j)-P(t,\alpha_{t}))X(t),X(t)\rangle (dN_{t}^{\alpha_{t}j}-q_{\alpha_{t}j}dt)\\
&+\int_{0}^{T}2\langle P(t,\alpha_t)[C(t,\alpha_t)X(t)+D(t,\alpha_t)u(t)],X(t)\rangle dW(t)+\int_{0}^{T}\langle\Lambda(t,\alpha_t)X(t),X(t)\rangle dW(t)\\,
\end{align*}
where $(N^{j'j})_{j',j\in\mathcal{M}}$ are independent Poisson processes each with intensity $q_{j'j}$ and $\widetilde{N}^{j'j}_{t}=N^{j'j}_{t}-q_{j'j}t,~t\geq0$ are the corresponding compensated Poisson martingales under the filtration $\mathcal{F}$. Since $X(t)$ is continuous, the Brownian martingales and Poisson martingales above are local martingales. Therefore there exists an increasing localizing sequence of stopping time $\tau_n\uparrow+\infty$ as $n\rightarrow+\infty$ such that
\begin{align*}
&\mathbb{E}\left[\left\langle G(\alpha_{T\wedge\tau_n})X(T\wedge\tau_n),X(T\wedge\tau_n)\right\rangle+\int_{0}^{T\wedge\tau_n}\left\langle
\left(
  \begin{array}{ll}
    Q(t,\alpha_{t}) & S(t,\alpha_{t})^{\top}\\
    S(t,\alpha_{t}) & R(t,\alpha_{t})
  \end{array}
\right)
\left(
  \begin{array}{c}
    X(t)\\
    u(t)
  \end{array}
\right),
\left(
  \begin{array}{c}
    X(t)\\
    u(t)
  \end{array}
\right)
\right\rangle dt\right]\\
=&\left\langle P(0,i_{0})x,x\right\rangle+\mathbb{E}\left[\int_{0}^{T\wedge\tau_n}\left\{\langle\left(R(t,\alpha_{t})+D(t,\alpha_{t})^{\top}P(t,\alpha_{t})D(t,\alpha_{t})\right)u(t),u(t)\rangle\right.\right.\\
&+2\langle\left(B(t,\alpha_{t})^{\top}P(t,\alpha_{t})+D(t,\alpha_{t})^{\top}P(t,\alpha_{t})C(t,\alpha_{t})+D(t,\alpha_{t})^{\top}\Lambda(t,\alpha_{t})+S(t,\alpha_{t})\right)X(t),u(t)\rangle\\
&\left.+\langle\left(PB+C^{\top}PD+\Lambda D+S^{\top}\right)\left(R+D^{\top}PD\right)^{-1}\left(B^{\top}P+D^{\top}PC+D^{\top}\Lambda+S\right)(t,\alpha_{t})X(t),X(t)\rangle\right\}dt\Bigg].
\end{align*}
Since for any $t\in[0,T],~i\in\mathcal{M}$, we have
\begin{align*}
R(t,i)+D(t,i)^{\top}P(t,i)D(t,i)>0,
\end{align*}
then it holds that
\begin{equation}\label{value}
\begin{aligned}
&\mathbb{E}\left[\left\langle G(\alpha_{T\wedge\tau_n})X(T\wedge\tau_n),X(T\wedge\tau_n)\right\rangle+\int_{0}^{T\wedge\tau_n}\left\langle
\left(
  \begin{array}{ll}
    Q(t,\alpha_{t}) & S(t,\alpha_{t})^{\top}\\
    S(t,\alpha_{t}) & R(t,\alpha_{t})
  \end{array}
\right)
\left(
  \begin{array}{c}
    X(t)\\
    u(t)
  \end{array}
\right),
\left(
  \begin{array}{c}
    X(t)\\
    u(t)
  \end{array}
\right)
\right\rangle dt\right]\\
\geq&\left\langle P(0,i_{0})x,x\right\rangle.
\end{aligned}
\end{equation}
We recall that $\mathbb{E}\left[\sup_{0\leq t\leq T}|X(t)|^{q}\right]<\infty$ for some some $q>2$. Now we let $n\rightarrow\infty$ and deduce from the dominated convergence and monotone convergence theorems that
\begin{align*}
&\mathbb{E}\left[\left\langle G(\alpha_{T})X(T),X(T)\right\rangle+\int_{0}^{T}\left\langle
\left(
  \begin{array}{ll}
    Q(t,\alpha_{t}) & S(t,\alpha_{t})^{\top}\\
    S(t,\alpha_{t}) & R(t,\alpha_{t})
  \end{array}
\right)
\left(
  \begin{array}{c}
    X(t)\\
    u(t)
  \end{array}
\right),
\left(
  \begin{array}{c}
    X(t)\\
    u(t)
  \end{array}
\right)
\right\rangle dt\right]\\
\geq&\left\langle P(0,i_{0})x,x\right\rangle.
\end{align*}
At last, we point out that the equality holds when
\begin{align*}
u(t,i)=&-\left(R(t,i)+D(t,i)^{\top}P(t,i)D(t,i)\right)^{-1}\\
&\cdot\left(B(t,i)^{\top}P(t,i)+D(t,i)^{\top}P(t,i)C(t,i)+D(t,i)^{\top}\Lambda(t,i)+S(t,i)\right)X(t,i),\quad t\in[0,T].
\end{align*}
\end{proof}

\end{document}